\documentclass[a4paper]{amsart}
\usepackage{amsmath, amssymb}
\usepackage[mathscr]{eucal}
\usepackage[colorlinks,pagebackref]{hyperref}
\pdfoutput=1

\newcommand{\A}{\mathbb{A}}

\newcommand{\C}{\mathbb{C}}

\newcommand{\G}{\mathbb{G}}
\newcommand{\bbH}{\mathbb{H}}
\newcommand{\M}{\mathbb{M}}
\newcommand{\N}{\mathbb{N}}
\newcommand{\Q}{\mathbb{Q}}
\newcommand{\R}{\mathbb{R}}
\newcommand{\Z}{\mathbb{Z}}

\newcommand{\Qp}{\mathbb{Q}_p}

\newcommand{\cA}{\mathcal{A}}
\newcommand{\cB}{\mathcal{B}}
\newcommand{\cC}{\mathcal{C}}
\newcommand{\cD}{\mathcal{D}}
\newcommand{\cE}{\mathcal{E}}
\newcommand{\cG}{\mathcal{G}}
\newcommand{\cH}{\mathcal{H}}
\newcommand{\cL}{\mathcal{L}}
\newcommand{\cM}{\mathcal{M}}
\newcommand{\cN}{\mathcal{N}}
\newcommand{\cO}{\mathcal{O}}
\newcommand{\cP}{\mathcal{P}}
\newcommand{\cS}{\mathcal{S}}
\newcommand{\cT}{\mathcal{T}}
\newcommand{\cV}{\mathcal{V}}

\newcommand{\scA}{\mathscr{A}}
\newcommand{\scB}{\mathscr{B}}

\newcommand{\scD}{\mathscr{D}}
\newcommand{\scE}{\mathscr{E}}
\newcommand{\scG}{\mathscr{G}}
\newcommand{\scH}{\mathscr{H}}
\newcommand{\scK}{\mathscr{K}}
\newcommand{\scM}{\mathscr{M}}
\newcommand{\scN}{\mathscr{N}}
\newcommand{\scP}{\mathscr{P}}

\newcommand{\scT}{\mathscr{T}}
\newcommand{\scU}{\mathscr{U}}
\newcommand{\scZ}{\mathscr{Z}}

\newcommand{\bB}{\mathbf{B}}

\newcommand{\bG}{\mathbf{G}}
\newcommand{\tbG}{\tilde\bG}

\newcommand{\bN}{\mathbf{N}}

\newcommand{\bT}{\mathbf{T}}
\newcommand{\bZ}{\mathbf{Z}}

\newcommand{\gh}{\mathfrak{h}}
\renewcommand{\gg}{\mathfrak{g}}
\newcommand{\gG}{\mathfrak{G}}
\newcommand{\gH}{\mathfrak{H}}
\newcommand{\gm}{\mathfrak{m}}
\newcommand{\gn}{\mathfrak{n}}
\newcommand{\gp}{\mathfrak{p}}
\newcommand{\gq}{\mathfrak{q}}
\newcommand{\gt}{\mathfrak{t}}

\newcommand{\an}{\mathrm{an}}
\newcommand{\cl}{\mathrm{cl}}
\DeclareMathOperator{\Cong}{Cong}
\newcommand{\cpct}{\mathrm{c}}
\newcommand{\cts}{\mathrm{cts}}
\DeclareMathOperator{\Eig}{Eig}

\DeclareMathOperator{\Ext}{Ext}
\newcommand{\fs}{\mathrm{fs}}
\DeclareMathOperator{\Gal}{Gal}
\newcommand{\group}{\mathrm{group}}
\newcommand{\GL}{\mathrm{GL}}
\newcommand{\tH}{\bar{H}}
\newcommand{\Hom}{\mathrm{Hom}}
\newcommand{\im}{\mathrm{Im}}
\newcommand{\ind}{\mathrm{ind}}
\newcommand{\Ind}{\mathrm{Ind}}
\newcommand{\locan}{\mathrm{la}}
\newcommand{\la}{\locan}
\newcommand{\lalg}{\mathrm{loc.alg.}}

\DeclareMathOperator{\Lie}{Lie}

\newcommand{\limp}[1]{\varprojlim_{#1}}
\newcommand{\limd}[1]{\varinjlim_{#1}}

\newcommand{\Null}{\mathrm{null}}
\newcommand{\ord}{\mathrm{ord}}
\newcommand{\pr}{\mathrm{pr}}

\newcommand{\ram}{\mathrm{ramified}}
\newcommand{\rep}{\mathrm{rep}}

\newcommand{\Rest}{\mathrm{Rest}}
\newcommand{\SL}{\mathrm{SL}}
\DeclareMathOperator{\slope}{slope}
\DeclareMathOperator{\Spec}{Spec}
\newcommand{\SO}{\mathrm{SO}}
\newcommand{\sph}{\mathrm{sph}}
\newcommand{\Spin}{\mathrm{Spin}}

\newcommand{\tors}{\mathrm{tors}}

\newcommand{\Repess}{\mathrm{Rep}^\mathrm{ess}}

\newtheorem{definition}{Definition}[subsection]
\newtheorem{theorem}[definition]{Theorem}
\newtheorem{proposition}[definition]{Proposition}
\newtheorem{lemma}[definition]{Lemma}
\newtheorem{corollary}[definition]{Corollary}

\theoremstyle{remark}
\newtheorem{remark}[definition]{Remark}
\newtheorem*{example}{Example}

\begin{document}

\title[$P$-adic interpolation of metaplectic forms]
{$P$-adic interpolation of metaplectic forms of cohomological type}

\author{Richard Hill}
\address{Department of Mathematics, University College London, \\ 
Gower Street, London WC1E 6BT, UK}
\email{rih@math.ucl.ac.uk}

\author{David Loeffler}
\address{Mathematics Institute, University of Warwick,\\
Coventry CV4 7AL, UK}
\email{d.a.loeffler@warwick.ac.uk}
\thanks{The second author's research has been supported by EPSRC Postdoctoral Fellowship EP/F04304X/2.}

\begin{abstract}
Let $\bG$ be a reductive group over a number field $k$.
It is shown how Emerton's methods may be applied to the problem of
$p$-adically interpolating the metaplectic forms on $\bG$, i.e. the
automorphic forms on metaplectic covers of $\bG$, as long as the metaplectic
covers involved split at the infinite places of $k$.
\end{abstract}



\maketitle
\section{Introduction}

\subsection{Metaplectic groups}

Let $k$ be an algebraic number field, $\A$ the ad\`ele ring of $k$, and $\bG$ a connected reductive
group over $k$.
Let $\mu$ be a finite abelian group. 
By a metaplectic cover of $\bG$ by $\mu$, we shall mean\footnote{
The notation $\tbG(\A)$ is standard but a little misleading, since this group is a topological group,
 but is not the group of adelic points of any linear group.} a topological central extension
\[
        1 \to \mu \to \tbG(\A) \to \bG(\A) \to 1,
\]
such that the subgroup $\bG(k)$ of rational points lifts to a subgroup $\hat\bG(k)$ of
$\tbG(\A)$.
We shall recall some results on the construction and properties of metaplectic covers in section 
\ref{sect:metgroups} below.

This paper is concerned with $p$-adic interpolation of automorphic representations of
a metaplectic group $\tbG(\A)$ (specifically, those representations which show up in the cohomology groups of its arithmetic quotients.)
More precisely we prove metaplectic versions of the results of \cite{emerton-interpolation},
\cite{emerton-jacquet1} and \cite{hill07}, which were originally proved there for automorphic representations of $\bG(\A)$.

We shall write $G_{\infty}$ for the Lie group $\bG(k \otimes_{\Q}\R)$ and $G_{\infty}^{\circ}$ for the
connected component of the identity in $G_{\infty}$.
We have a topological central extension of real Lie groups
\begin{equation}
        \label{archimedean}
        1 \to \mu \to \tilde G_{\infty}^{\circ} \to G_{\infty}^{\circ} \to 1,
\end{equation}
where $\tilde G_{\infty}^{\circ}$ is the pre-image of $G_{\infty}^{\circ}$ in $\tbG(\A)$.
We shall divide metaplectic covers into two kinds:
\begin{itemize}
        \item
        \textbf{Type 1}: The extension \eqref{archimedean} splits as a direct
        sum;
        \item
        \textbf{Type 2}: The extension \eqref{archimedean} does not split.
\end{itemize}
In this article we shall be concerned with metaplectic groups of type 1. As we 
point out in Proposition \ref{prop2.2} below, there are many groups $\bG$ for 
which all metaplectic covers of $\bG$ are of type 1, so this restriction is not 
too onerous. Our methods will not apply to type 2 covers, since such covers do 
not possess genuine metaplectic forms of cohomological type. However we note 
that various authors have dealt with the $p$-adic interpolation of metaplectic 
forms of type 2 in certain cases, using very different methods (see for example 
\cite{koblitz86,ramsey-2006,ramsey-2008}).

\subsection{Cohomology of arithmetic quotients}

Assume now that we have a type 1 metaplectic cover of $\bG$.
We fix once and for all a lift $\hat \bG(k)$ of $\bG(k)$ to $\tbG(\A)$.
We also fix a maximal compact subgroup $K_\infty^\circ$ of $G_\infty^{\circ}$.
For a compact open subgroup $K_{f}$ of $\bG(\A_{f})$, we write $Y(K_{f})$ for the corresponding
arithmetic quotient of $\bG$, i.e.
\[
        Y(K_{f})
        =
        \bG(k)\backslash\bG(\A)/K_{\infty}^{\circ}K_{f}.
\]
We shall describe $Y(K_{f})$ in a different way, which is more useful to us.
Suppose that $K_{f}$ is chosen small enough so that it lifts
to a subgroup $\hat K_{f}$ of $\tbG(\A)$.
The existence of such a $K_{f}$ follows from topological considerations, and we
shall fix such a lift $\hat K_{f}$.

As $\tilde\bG$ is of type 1, the group $K_{\infty}^{\circ}$
lifts uniquely to a maximal compact subgroup 
$\hat K_{\infty}^{\circ}$ of $\tilde G_\infty^{\circ}$.
Since $\hat K_{\infty}^{\circ}$ is connected, it follows that $\hat K_{f}$ commutes with $\hat K_{\infty}^{\circ}$,
and so we have the following alternative description of $Y(K_{f})$:
\[
        Y(K_{f})
        =
        \tilde\bG(k)\backslash\tbG(\A)/\hat K_{\infty}^{\circ}\hat K_{f},
\]
where $\tilde\bG(k)$ is the pre-image of $\bG(k)$ in $\tbG(\A)$.

This allows us to define a local system on $Y(K_f)$ for any representation of 
$\tilde\bG(k) \cong \bG(k) \oplus \mu$. In particular, if $W$ is a 
representation of $\bG(k)$ (over some coefficient field $E$) and 
$\varepsilon: \mu \to E^\times$ is a character, we may consider the 
representation $W \otimes \varepsilon$ of $\tbG(k)$. The corresponding local
system $\cV_{W\otimes \varepsilon}$ is given by
\[
        \cV_{W\otimes\varepsilon}
        =
        \tilde \bG(k)\backslash \left((\tbG(\A)/\hat K_{\infty}^{\circ} \hat K_{f})\times (W\otimes\varepsilon) \right).
\]

The cohomology groups of these local systems have an action of the Hecke algebra 
of $\tbG(\A_f)$ with respect to $\hat K_f$. This article will be concerned with 
the interpolation of the systems of Hecke eigenvalues appearing in these 
cohomology groups. More precisely, we shall fix a prime $\gp$ of $k$ above the 
rational prime $p$, and take $W$ to be an algebraic representation of the group 
$\cG=\Rest^{k_{\gp}}_{\Qp}(\bG \times_{k} k_{\gp})$,
where $\Rest^{k_{\gp}}_{\Qp}$ denotes restriction of scalars.
We regard such a
representation $W$ as a representation of $\bG(k)$ via the inclusion 
$\bG(k) \hookrightarrow \bG(k_\gp) \cong \cG(\Qp)$. Our goal is to show that
the Hecke eigenvalues appearing in these spaces move in $p$-adic families as the
weight of $W$ varies.

We will find it more convenient to work with the corresponding smooth representations.
Let us write $\scG$ for the group $\cG(\Q_{p})$. For a prime $\gq$ of $k$ we shall often write
$G_{\gq}$ for the group $\bG(k_{\gq})$. Thus there is a canonical identification $\scG=G_{\gp}$.
For a subgroup $U\subset \bG(\A)$, we shall write $\tilde U$ for the preimage of $U$ in $\tbG(\A)$.
If $U$ is a subgroup of $K_{f}$, then we shall write $\hat U$ for the lift of $U$ to $\hat K_{f}$,
so that $\hat U = \tilde U \cap \hat K_{f}$.
We shall assume that $K_{f}=K_{\gp}K^{\gp}$, where $K_{\gp}$ is a compact open subgroup
of $G_{\gp}$ and $K^{\gp}$ is a compact open subgroup of $\bG(\A^{\gp}_{f})$.
The group $K^{\gp}$ will be called the ``tame level'' and will remain fixed throughout the paper. We thus have a corresponding decomposition $\hat K_f = \hat K_{\gp} \hat K^{\gp}$. 

We define for a fixed tame level:
\[
        H^{\bullet}_{\cl, \varepsilon}(\hat K^{\gp},W)
        :=
        \limd{U \subset K_\gp} H^{\bullet}(Y(UK^{\gp}),\cV_{W\otimes\varepsilon}).
\]
Let $\cH^{\gp}$ be the Hecke algebra of $\tbG(\A_{f}^{\gp})$ with respect to
$\hat K^{\gp}$, and with coefficients in $E$.
The vector spaces $H^{\bullet}_{\cl,\varepsilon}(\hat K^{\gp},W)$ have commuting actions of
 $\tilde \scG=\tilde G_{\gp}$ and $\cH^{\gp}$.
As a $\tilde \scG$-module, $H^{\bullet}_{\cl, \varepsilon}(\hat K^{\gp},W)$
is smooth and admissible.
We may recover the finite level cohomology $H^{\bullet}(Y(K^{\gp}K_{\gp}),\cV_{W\otimes\varepsilon})$
as the subspace of $\hat K_{\gp}$-invariants in $H^{\bullet}_{\cl, \varepsilon}(\hat K^{\gp},W)$.

\subsection{Classical points and Eigenvarieties}

We may replace $\mu$ by $\mu/\ker\varepsilon$ without changing the
classical cohomology groups. We shall therefore assume from now on that the map 
$\varepsilon:\mu \to E^{\times}$ is injective. We shall also assume that the 
coefficient field $E$ is a complete and discretely valued subfield of $\C_p$, 
and that $\cG$ is split over $E$.

Our results are most complete (and easiest to state) when $\cG$ is quasi-split
over $\Qp$. Let us assume for now that this is the case, and choose a Borel 
subgroup $\cB=\cT\cN$ of $\cG$ defined over $\Qp$, with unipotent radical $\cN$ 
and Levi factor $\cT$.
The group $\scN$ of $\Qp$-valued points of $\cN$ has a unique lift $\hat \scN$ to $\tilde \scG$.
For a smooth representation $V$ of $\tilde \scG$, the smooth Jacquet module $V_{\hat \scN}$
is defined to be the space of $\hat \scN$-coinvariants.
This Jacquet module is a smooth representation of the group $\tilde \scT$, which is
the pre-image in $\tilde\scG$ of the group $\scT$ of $\Qp$-valued points of $\cT$.
Applying this smooth Jacquet functor to $H^{\bullet}_{\cl, \varepsilon}(\hat K^{\gp},W)$,
we obtain a representation $H^{\bullet}_{\cl, \varepsilon}(\hat K^{\gp},W)_{\hat \scN}$
of $\tilde \scT \times \cH^{\gp}$, which is smooth and admissible as a representation of $\tilde\scT$.
If $W$ is irreducible, then the contragredient representation $W'$ is also irreducible,
and hence by the highest weight theorem (cf.~\cite[Theorem 1.1.1]{emerton-jacquet1})
$(W')^{\cN}$ is a 1-dimensional algebraic representation
of $\cT$.
Tensoring with this representation, we obtain a locally algebraic representation
$\rep^{s}(W)=\rep^{s}(\hat K^{\gp},\varepsilon,W)$ of $\tilde \scT\times \cH^{\gp}$ defined as follows:
\[
        \rep^{s}(W)
        =
        (W')^{\cN} \otimes_{E} H^{s}_{\cl, \varepsilon}(\hat K^{\gp},W)_{\hat \scN}.
\]

By a \emph{classical point} of cohomological dimension $s$, we shall mean an absolutely irreducible
(and hence finite dimensional) subquotient of $\rep^{s}(W)$ for some $W$.
Under a certain hypothesis (see definition \ref{edgemapcriterion} below) we shall construct a rigid analytic space,
called the eigenvariety, containing all the classical points (independently of the choice of irreducible representation $W$).

Each classical point is of the form $\pi=\pi_{\gp}\otimes \pi^{\gp}$, where $\pi_{\gp}$ is an irreducible locally algebraic representation of $\tilde \scT$ and $\pi^{\gp}$ is an irreducible representation of $\cH^{\gp}$.
Let ${Z}$ be the centre of $\tilde \scT$.
Recall that by the Stone-von Neumann Theorem for smooth representations of metaplectic tori \cite[Theorem 3.1]{weissman09}, the representation $\pi_{\gp}$ is determined by its restriction to ${Z}$ (here we have used the assumption that $\varepsilon$ is injective).
We may therefore regard $\pi_{\gp}$ as a point of the rigid analytic space $\widehat{{Z}}$ of locally analytic characters of ${Z}$.

There is a decomposition $\cH^{\gp}=\cH^{\sph}\otimes \cH^{\ram}$, where $\cH^{\sph}$ is commutative
and $\cH^{\ram}$ is finitely generated over $E$.
This gives us a further decomposition $\pi^{\gp}=\pi^{\sph}\otimes \pi^{\ram}$, where $\pi^{\sph}$ is
in $\Spec(\cH^{\sph})$.
Thus a classical point gives rise to a point $(\pi_{\gp},\pi^{\sph})$ of $\widehat Z \times \Spec(\cH^{\sph})$.
We shall show that there is a rigid analytic subspace
$\Eig^{s} \subset \widehat {{Z}} \times \Spec(\cH^{\sph})$
containing all the classical points of cohomological dimension $s$,
such that the projection $\Eig^{s} \to \widehat Z$ is finite, with discrete fibres.
In particular, the dimension of $\Eig^{s}$ is at most the dimension of $\cT$.

We shall see that there is a representation
$J^{s}=J^{s}(\hat K^{\gp},\varepsilon)$ of $\tilde \scT\times \cH^{\gp}$, together with injective homomorphisms
$\rep^{s}(W) \to J^{s}$ for each $W$, so that the representation $J^{s}$ interpolates
all of the classical points.
As a representation of $\tilde \scT$, $J^{s}$ is locally analytic, and we shall prove
that $J^{s}$ is essentially admissible in the sense of \cite{emerton-memoir}
(see definition \ref{essen-admiss} below).
For such a representation, there is a corresponding coherent sheaf $\cE$ on $\widehat Z$,
such that the fibre at $\chi\in \widehat Z$ is isomorphic (as an $\cH^{\gp}$-module)
 to the contragredient of the $\chi$-eigenspace in $J^{s}$.
The eigenvariety is defined to be the relative spectrum $\Eig^{s}=\Spec(\cA)$,
where $\cA$ is the image of $\cH^{\sph}$ in
the sheaf of endomorphisms of $\cE$.
Localizing $\cE$ over $\cA$, we obtain a canonical sheaf $\cM$ on $\Eig^{s}$.
The sheaf $\cM$ allows us to recover the action of $\cH^{\ram}$ on $J^{s}$.

Note that each classical point on the eigenvariety has a $\widehat {{Z}}$-coordinate which is locally algebraic.
It would be useful to know whether the converse is true, i.e. given a point of the eigenvariety whose $\widehat{{Z}}$-coordinate is locally algebraic, can we conclude that the point is a classical point?
We prove that if $\chi\in\widehat {{Z}}$ is locally $(W')^{\cN}$-algebraic and has \emph{non-critical slope}
(see definition \ref{non-crit-slope} below),
then the map of $\chi$-eigenspaces $\rep^s(W)[\chi] \to J^{s}[\chi]$ is bijective.
Hence every point on the eigenvariety with $\widehat {{Z}}$-coordinate $\chi$ is classical.

The representation $J^{s}$ together with the maps $\rep^{s}(W)\to J^{s}$ are defined in two steps:
completed cohomology and the locally analytic Jacquet functor.
We shall briefly discuss these constructions in the next two paragraphs of this introduction.

\subsection{Completed cohomology}

Recall that we have fixed a prime $\gp$ of $k$ above $p$. Let $\cO_E$ be the 
valuation ring of $E$.

We define the \emph{completed cohomology spaces} $\tH^{s}=\tH^{s}(\hat K^{\gp},E)$ 
as follows:\footnote{The definition of these spaces generalises the spaces 
denoted by $\tilde{H}^s(...)$ in \cite{emerton-interpolation}. We denote them 
here with a horizontal line rather than a tilde in order to avoid conflict with 
our use of the tilde sign to denote preimages in $\tbG(\A)$ of subgroups of 
$\bG(\A)$.} The choice of $\hat K_f$ determines a $\mu$-covering space 
$\tilde Y(\hat K_{f})$ of $Y(K_{f})$, defined by
\[
        \tilde Y(\hat K_{f})
        =
        \hat\bG(k)\backslash\tbG(\A)/\hat K_{\infty}^{\circ}\hat K_{f}.
\]
We set
\[
        \tH^{s}(\hat K^{\gp}, E)
        =
        \left(\limp{n}\limd{K_{\gp}} H^{s}(\tilde Y(\hat K^{\gp}\hat K_{\gp}),\cO_E / p^{n}) \right)
        \otimes_{\cO_E} E.
\]
The completed cohomology spaces are naturally Banach spaces over $E$,
with commuting actions of $\tilde \scG = \tilde G_{\gp}$ and the Hecke algebra 
$\cH^{\gp}$. We show in section \ref{sect:htilde} that $\tH^{s}$ is an 
admissible continuous representation of $\tilde \scG$ in the sense of 
\cite{ST-banach} (see definition \ref{cts-admiss} below). We shall write 
$\tH^{s}_{\varepsilon}(\hat K^{\gp}, E)$ for the $\varepsilon'$-eigenspace for 
the action of $\mu$ on $\tH^{s}(\hat K^{\gp}, E)$.

Let $\gg$ be the Lie algebra of $\cG$ over $\Qp$.
There is a natural action of $\gg$ on the subspace $\tH^{s}_{\varepsilon,\la}$
of $\Qp$-locally analytic vectors in $\tH^{s}_{\varepsilon}$.
We show (see Corollary \ref{spectral-sequence})
that there is a spectral sequence relating this action to the classical cohomology spaces
for any algebraic representation $W$:
\[
        \Ext^{r}_{\gg}(W' , \tH^{s}_{\varepsilon,\la})
        \Rightarrow
        H^{r+s}_{\cl, \varepsilon}(\hat K^{\gp}, W).
\]
Here $W'$ denotes the contragredient of $W$.
In particular, we have an edge map
\begin{equation}
        \label{edgemap}
        H^{s}_{\cl,\varepsilon}(\hat K^{\gp},W)
        \to
        \Hom_{\gg}(W', \tH^{s}_{\varepsilon,\la}).
\end{equation}
This map is a homomorphism of smooth $\tilde \scG\times \cH^{\gp}$-representations.
Written another way, this gives a homomorphism of locally analytic representations:
\begin{equation}
        \label{edgemap2}
        W'\otimes H^{s}_{\cl,\varepsilon}(\hat K^{\gp},W)
        \to
        \tH^{s}_{\varepsilon,\la}.
\end{equation}

\begin{definition}
        \label{edgemapcriterion}
        We say the tuple $(\tilde\bG,\gp, \varepsilon,\hat K^{\gp},s)$ satisfies
        \emph{Emerton's edge map criterion}
        if for every $W$ the map \eqref{edgemap} is an isomorphism,
        or equivalently if the image of \eqref{edgemap2} is the set
        of all $W'$-locally algebraic vectors in $\tH^{s}_{\varepsilon}$.
\end{definition}

Our results concerning eigenvarieties described above depend on the edge map criterion.
It is therefore rather important to know that this criterion holds in a number of cases. It is clear that the criterion is satisfied for $s = 0$. We prove the following result for $s = 1$ and $s = 2$, generalizing the results of \cite{hill07} to the metaplectic case.

\begin{theorem}
        Suppose $\bG$ is semi-simple and simply connected and has positive real 
        rank (i.e.~$G_\infty$ is not compact). Then the edge map \eqref{edgemap} 
        is an isomorphism in dimension $s=1$. If in addition $\bG$ has finite 
        congruence kernel and $\varepsilon$ is non-trivial then the edge map is 
        an isomorphism in dimension $s=2$.
\end{theorem}

\subsection{The locally analytic Jacquet module}

We then turn to finding an analogue in this situation of the locally analytic
Jacquet module construction of \cite{emerton-jacquet1}.
Let $\cP$ be a parabolic subgroup of $\cG$ defined over $\Qp$,
with unipotent radical $\cN$ and Levi factor $\cM$.
In line with our earlier notation, we shall write $\scP$, $\scN$ and $\scM$ for the
groups of $\Qp$-valued points of $\cP$, $\cN$ and $\cM$ respectively.
For each such $\cP$, we define a left-exact functor $J_{\cP}$
from locally analytic representations of $\tilde \scP$ to those of
$\tilde \scM$.

Any smooth or locally algebraic representation may be regarded as a locally analytic
representation, and so we may apply $J_{\cP}$ to such representations.
We shall prove the following, which determines $J_{\cP}$ on locally algebraic representations:

\begin{theorem}
        If $X$ is a smooth admissible representation of $\tilde \scG$
        and $W$ is an algebraic representation of $\cG$, then
        \[
                J_{\cP}(X \otimes W)
                =
                X_{\hat \scN} \otimes W^\cN,
        \]
        where $\hat \scN$ is the unique lift to $\tilde \scG$ of $\scN$.
\end{theorem}

In particular, the locally analytic Jacquet functor coincides with the smooth Jacquet functor on smooth representations.

Suppose again that $\cG$ is quasi-split over $\Qp$ and let $\cB$ be a Borel subgroup with Levi subgroup $\cT$.
The representation $J^{s}$ discussed above is defined as follows:
\[
        J^{s} = J_{\cB}\left(\tH^{s}_{\varepsilon}(\hat K^{\gp},E)_{\la}\right).
\]
Applying the Jacquet functor $J_{\cB}$ to the map \eqref{edgemap2} we get the required map
$\rep^{s}(W) \to J^{s}$. 
Assuming the edge map criterion, we know that \eqref{edgemap2} is injective with closed image.
Hence by left exactness, we conclude that the map $\rep^{s}(W)\to J^{s}$ is injective.

Roughly speaking, the eigenvariety $\Eig^s$ defined above is the set of characters of $\widehat{{Z}} \times \mathcal{H}^{\sph}$ appearing in $J^s$, where $\mathcal{H}^{\sph}$ is the spherical part of the Hecke algebra of $\hat K^{\gp}$. If the edge map criterion holds, then since $J^s$ contains $\rep^s(W)$ for every $W$, this space contains all characters arising from automorphic representations of $\tbG(\A)$ which are cohomological in degree $s$, which have a $\hat K^{\gp}$-fixed vector, and whose local factor at $\gp$ is principal series. We also give a ``classicality criterion'', i.e. a sufficient condition for a point $(\chi, \lambda)$ of $\Eig(J^s)$ to appear in the classical cohomology $\rep^s(W)$.

The completed cohomology $\tH^{s}_{\varepsilon}$ is a continuous admissible representation of $\tilde \scG$.
From this, it follows that its subspace $\tH^{s}_{\varepsilon,\la}$ of locally analytic vectors
is a strongly admissible locally analytic representation of $\tilde \scG$ in the sense of \cite{emerton-memoir}.
In order to show that $J^{s}$ is essentially admissible, we prove the following.

\begin{theorem}
        If $V$ is an essentially admissible locally analytic representation of $\tilde \scG$,
        then $J_{\cP}(V)$ is an essentially admissible locally analytic representation of $\tilde \scM$.
\end{theorem}

\subsection{Representations of metaplectic tori}

A novel aspect of the metaplectic case, compared to the theory for algebraic groups, is that the representations $J^s$ constructed in the above fashion are representations of a non-abelian group (a metaplectic extension of a torus). Essentially admissible locally analytic representations of abelian locally analytic groups are well-understood, and may be interpreted as coherent sheaves on a rigid-analytic space, which is the method used in \cite{emerton-interpolation} in order to construct eigenvarieties. 

In the method described above, we have restricted our representation $J^{s}$ of $\tilde \scT$
to the centre ${{Z}}$ of $\tilde \scT$ in order to construct the eigenvariety.
This construction has certain drawbacks.
The first problem is to do with the ramified part of the Hecke algebra.
Let us suppose that we have an absolutely irreducible $\tilde\scT\times \cH^{\gp}$-subrepresentation
$\pi_{\gp}\otimes\pi^{\sph}\otimes \pi^{\ram}$ of $J^{s}$,
and let $\chi:Z\to E^{\times}$ be the central character of $\pi_{\gp}$.
We have a corresponding point $(\chi,\pi^{\sph})$ in the eigenvariety.
One would ideally like recover the representation $\pi^{\ram}$ in the dual space
of the fibre of the sheaf $\cM$ at the point $(\chi,\pi^{\sph})$.
However, with the construction described above the contribution to the
fibre is $(\pi^{\ram})^{*d}$, where $d$ is the dimension of $\pi_{\gp}$.
In a sense, this means that the sheaf $\cM$ is $d$ times as big as we would like.

The second drawback is that the field of definition of a point of the eigenvariety is not exactly the
same as the field of definition of the corresponding absolutely irreducible representation.
More precisely, suppose that $\chi:{Z} \to E^{\times}$ is a locally analytic character, and let $\pi_{\gp}$
be the corresponding absolutely irreducible locally analytic representation of $\tilde\scT$.
It can quite easily happen that $\pi_{\gp}$ is not defined over $E$, but only over some finite extension.
The field of definition of a point on the eigenvariety will only see the field of definition of $\chi$, since
we have restricted to ${Z}$ in our construction.
To some extent this problem is unavoidable, since $\pi_{\gp}$ will often have no unique
minimal field of definition.

In section \ref{sect:stonevonneumann}, we show that both of these problems may be resolved
assuming a certain ``tameness'' condition on the group $\tilde\scT$.
Assuming the tameness condition,  we show that there is an equivalence of categories
between the essentially admissible locally analytic representations of $\tilde\scT$ extending $\varepsilon$,
and the essentially admissible locally analytic representations of $Z$ extending $\varepsilon$.
In particular, this implies that that character $\chi$ and the corresponding representation
$\pi_{\gp}$ have the same field of definition.
Applying this equivalence of categories to $J^{s}$ instead of simply restriction of representations,
we obtain a slightly different coherent sheaf $\cE$ on $\widehat{Z}$. This gives rise to a slightly
different sheaf $\cM$ on the eigenvariety, and this new sheaf has the correct multiplicities of
representations of $\cH^{\ram}$.
Finally, we show that if $\bG$ is semi-simple, simply connected and split over $k_{\gp}$ and if
$p$ does not divide the order of $\mu$, then $\tilde\scT$ satisfies the tameness condition.


\subsection{Relation to the work of Emerton}

In this paper we rely heavily on the work of Emerton. The results and definitions of this paper are analogues of those obtained by Emerton in the case of algebraic groups
(as opposed to metaplectic groups). The locally analytic Jacquet functor for representations of reductive algebraic groups was introduced in \cite{emerton-jacquet1}, and completed cohomology was introduced in \cite{emerton-interpolation}, where it was used to construct eigenvarieties for algebraic groups.
In many parts of this paper -- particularly in section \ref{sect:jacquet} -- the proofs of our results very closely follow those of Emerton, and rather than reproducing the intricate proofs in full, we have simply indicated how the original arguments need to be modified in order to apply to the metaplectic case.

%
%
%

\section{Classical cohomology of metaplectic groups}

In this section, we shall recall some standard results on metaplectic groups, and recall the construction of admissible smooth representations arising from the cohomology of arithmetic quotients of these groups.

\subsection{Metaplectic groups}
\label{sect:metgroups}

As before, we let $\bG$ be a connected reductive group over an algebraic number field $k$.
It is shown in \cite{deligne96} that if $k$ contains a primitive $m$-th root of unity,
then there is a canonical non-trivial metaplectic extension of $\bG$ by the group
$\mu_{m}$ of all $m$-th roots of unity in $k$, and also a canonical lift $\hat \bG(k)$.
If $\bG$ is absolutely simple and algebraically simply connected then there is a universal metaplectic extension,
whose kernel is the group of all roots of unity in $k$.
The universal metaplectic cover coincides with Deligne's cover.
However for our purposes, it is sufficient to choose a metaplectic cover, together with a lift $\hat\bG(k)$.

Let $K_{\infty}^{\circ}\subset G_{\infty}^{\circ}$ be a maximal compact subgroup.
We have a topological central extension of compact Lie groups:
\begin{equation}
        \label{compact}
        1 \to \mu \to \tilde K_{\infty}^{\circ} \to K_{\infty}^{\circ} \to 1.
\end{equation}

Recall that $G_{\infty}^\circ$ has an Iwasawa decomposition as the product of $K_\infty^\circ$ and a uniquely divisible, topologically contractible group $H$. This group $H$ therefore has a unique lift to a subgroup $\hat H$ of $\tilde G_\infty^\circ$, so we have a corresponding Iwasawa decomposition $\tilde G_{\infty}^\circ = \tilde K_\infty^\circ \hat H$. Thus the inclusions $K_{\infty}^{\circ}\hookrightarrow G_{\infty}^{\circ}$ and
$\tilde K_{\infty}\hookrightarrow\tilde  G_{\infty}^{\circ}$ are homotopy equivalences.
It follows that the extension \eqref{archimedean} splits
 if and only if \eqref{compact} splits.


\begin{example}
        Let $\bG=\SL_{n}/\Q$.
        There is a unique non-trivial metaplectic double cover $\widetilde{\SL}_{n}(\A)$.
        We may take $K_{\infty}=\SO(n)$, and then $\tilde K_{\infty}=\Spin(n)$.
        The extension is of type 2, since $\Spin(n)$ is connected.
\end{example}

The following proposition gives an ample supply of type 1 metaplectic covers.

\begin{proposition}
        \label{prop2.2}
        Let $\bG$ be semi-simple and algebraically simply connected, and suppose that for every
        real place $v$ of $k$, the group $G_{v}$ is compact.
        Then every metaplectic extension of $\bG$ is of type 1.
        In particular this holds if $k$ is totally complex.
\end{proposition}

\begin{proof}
        We will show that under these hypotheses, $G_v$ is topologically
        simply connected for each infinite place $v$.

        If $v$ is a complex place of $k$ then $G_{v}$
        is a complex algebraically simply connected group. By the Iwasawa
        decomposition, $G_v$ is homotopy equivalent to a maximal compact
        subgroup. On the other hand, if $k_v$ is real and $G_v$ is compact,
        then $G_v$ is itself a maximal compact subgroup of the
        complexification $\bG(\mathbb{C})$. So it suffices to show that if $\bG$
        is an algebraically simply connected semi-simple Lie group over
        $\C$, then any maximal compact subgroup of $\bG(\mathbb{C})$ is
        topologically simply connected. This follows readily from Theorem 1.1 of
        \cite{adams04}.

        The group $ G_{\infty}$ is therefore a product of
        simply connected spaces, so is simply connected.
        It follows that the extension splits over $ G_{\infty}$.
\end{proof}

\begin{remark}
        It is a widely held misconception that when $k$ is totally complex every metaplectic extension
        is of type 1.
        By the above proposition, this holds when $\bG$ is semi-simple and simply connected,
        but it is false in general.
        Indeed it is even false for double covers of $\GL_{1}$ (the extension constructed in \cite{hill-OJAC}
        is a counterexample).
\end{remark}

\subsection{Arithmetic quotients}

We assume for the remainder of this paper that $\tbG$ is of type 1.

Recall that for a compact open subgroup $K_f$ of $\bG(\A_f)$ we have defined
an arithmetic quotient $Y(K_{f})$.
Assuming that $K_{f}$ has a lift $\hat K_{f}$ to $\tilde\bG(\A)$, we have defined
a $\mu$-covering space $\tilde Y(\hat K_{f})$ of $Y(K_{f})$.
The topological spaces $Y(K_{f})$ and $\tilde Y(\hat K_{f})$ are homotopic to
finite simplicial complexes (see \cite{borel-serre}).
If $K_{f}$ is sufficiently small, then these are topological manifolds.

As described in the introduction, for any finite-dimensional algebraic
representation $W$ of $\cG$, over some field $E$ containing $\Qp$, and any
character $\varepsilon: \mu \to E^\times$, we have a locally constant sheaf of
$E$-vector spaces on $Y(K_f)$,
\[
        \cV_{W\otimes\varepsilon}
        =
        \tilde \bG(k)\backslash \left((\tbG(\A)/\hat K_{\infty}^{\circ} \hat K_{f})\times (W\otimes\varepsilon) \right).
\]
Since $W$ is finite-dimensional and $Y(K_f)$ is homotopic to a finite simplicial
complex, the cohomology groups
$H^\bullet(Y(K_f), \cV_{W\otimes\varepsilon})$ are finite-dimensional $E$-vector spaces.

The formation of $\cV_{W \otimes \varepsilon}$ is compatible with pullback via
the natural maps $Y(K_f') \to Y(K_f)$, for $K_f' \subseteq K_f$. Moreover, for
$g \in \tbG(\A_f)$, right translation defines an isomorphism $[g] : Y(K_f) \to
Y(g^{-1} K_f g)$, and an isomorphism of local systems $[g]^* \cV_{W \otimes
\varepsilon} \cong \cV_{W \otimes \varepsilon}$; when $g \in \mu$, we have
$Y(g^{-1} K_f g) = Y(K_f)$, and the map is simply multiplication by
$\varepsilon(g)$. These compatibilities imply that if $K^{\gp}$ is a tame level,
then the spaces 
\[
        H^{\bullet}_{\cl, \varepsilon}(\hat K^{\gp},W)
        :=
        \limd{U \subset K_f} H^{\bullet}(Y(U K^{\gp}),\cV_{W \otimes \varepsilon}).
\]
are smooth representations of $\tilde G_\gp$, on which $\mu$ acts via the
character $\varepsilon$. Since the $\hat K_{\gp}$-invariants of the
representation $H^{\bullet}_{\cl, \varepsilon}(\hat K^{\gp},W)$ can be
identified with $H^{\bullet}(Y(K_{\gp} K^{\gp}),\cV_{W\otimes\varepsilon})$,
the representations $H^{\bullet}_{\cl, \varepsilon}(\hat K^{\gp},W)$ are
admissible smooth representations of $\tilde G_{\gp}$.

Note that we also have a local system on $\tilde Y(\hat K_{f})$ defined by
\[
        \cV_{W}
        =
        \hat \bG(k)\backslash \left((\tbG(\A)/\hat K_{\infty}^{\circ} \hat K_{f})\times W \right).
\]
If we write $\pr: \tilde Y(\hat K_{f}) \to Y(K_{f})$ for the projection map,
then we have an isomorphism of local systems:
\[
        \pr_{*}\left(\cV_{W}\right)
        =
        \bigoplus_{\eta:\mu\to E^{\times}}
        \cV_{W\otimes\eta},
\]
and hence by Shapiro's lemma (or the spectral sequence of the map $\pr$),
\[
        H^{\bullet}(\tilde Y(\hat K_{f}),\cV_{W})
        =
        \bigoplus_{\eta:\mu\to E^{\times}}
        H^{\bullet}(Y(K_{f}),\cV_{W\otimes\eta}).
\]

\subsection{Connection with the Kubota symbol}
\label{sect:kubota}

We give an alternative description of the cohomology groups in the special case
that $\bG$ is absolutely simple, simply connected, and has positive real rank.
In this situation there is a universal metaplectic cover $\tbG(\A)$ by the group
$\mu_{m}$ of all roots of unity in $k$. The groups
$G_{\infty}$ and $K_{\infty}$ are connected, and their quotient $X= G_{\infty}
/ K_{\infty}$ is a symmetric space. The compact open subgroup $K_{f}$ determines
an arithmetic subgroup
\[
        \Gamma = \bG(k) \cap  \left( K_{f}\times  G_{\infty}\right).
\]
Furthermore the arithmetic quotient $Y(K_{f})$ is connected, and may be
identified with a quotient of $X$ as follows:
\[
        Y(K_{f})
        =
        \Gamma \backslash X.
\]
Recall that we have lifts $\tau_{1}:\bG(k) \to \hat \bG(k)$, and
$\tau_{2}: K_{f} \times G_{\infty} \to  \hat K_{f} \times\hat G_{\infty}$.
These lifts are both defined on the arithmetic subgroup $\Gamma$, but they are
not equal on that subgroup.
The \emph{Kubota symbol} is defined to be the ratio of these two lifts:
\[
        \kappa(\gamma) = \tau_{1}(\gamma)\tau_{2}(\gamma)^{-1},
        \qquad
        \gamma \in \Gamma.
\]
We recall that the Kubota symbol is a surjective homomorphism $\Gamma \to
\mu_{m}$. Its kernel is a non-congruence subgroup of $\Gamma$. Indeed in many
cases $\kappa$ gives an isomorphism between the congruence kernel and $\mu_{m}$
(see for example Theorem 2.9 of \cite{prasad-raghunathan-congruence}).

Suppose again that $W$ is an algebraic representation of $\bG$ over $E$. By
restriction, we obtain an action of $\Gamma$ on $W$, and we may twist this
action by the character $\varepsilon\circ\kappa$ to get a new action. We can
form the local system on $Y(K_{f})$:
\[
        \cV'_{W\otimes\varepsilon}= \Gamma \backslash \Big(X \times (W\otimes(\varepsilon\circ\kappa))\Big).
\]
One can check that $\cV'_{W\otimes\varepsilon}$ is an isomorphic local system to $\cV_{W\otimes\varepsilon}$.
In particular, we may express our cohomology groups in terms of group cohomology:
\[
        H^{\bullet}(Y(K_{f}), \cV_{W\otimes\varepsilon})
        =
        H^{\bullet}_{\group}(\Gamma,W \otimes(\varepsilon\circ\kappa)).
\]

\subsection{Non-vanishing of metaplectic cohomology}

We next show in some simple cases that the spaces $H^{\bullet}_{\cl,\varepsilon}(\hat K^{\gp},W)$
 are non-zero.

\begin{proposition}
        Suppose that $ G_{\infty}$ is compact.
        Then for any $K_f$ sufficiently small and any $W$, the
        vector space $H^{0}_{\cl,\varepsilon}(\hat K_{f},W)$ is non-zero.
\end{proposition}

\begin{proof}
Since $G_\infty$ is compact, the double quotient
\[ \tbG(k) \backslash \bG(\mathbb{A}) / \hat K_\infty^{\circ} \hat K_f\] 
is a finite set. Moreover, if $\mu_1, \dots, \mu_r$ is a set of coset
representatives, the groups
\[ \Gamma_j = \tbG(k) \cap \mu_j \hat K_\infty^{\circ} \hat K_f \mu_j^{-1}\]
are finite. Then, for any $W$, the space $H^0(Y(K_f), \cV_{W\otimes\varepsilon})$ can be identified
with the space of maps from the finite set $\mu_1, \dots, \mu_r$ to $W\otimes \varepsilon$ for
which $f(\mu_j) \in (W\otimes\varepsilon)^{\Gamma_j}$. By shrinking $K_f$ if necessary, we may
assume that $\hat K_f$ is torsion-free, and hence all the groups $\Gamma_j$ are
trivial. Therefore $H^0(Y(K_f), \cV_{W\otimes\varepsilon})$ is non-zero.
\end{proof}

\begin{proposition}
        \label{non-vanish2}
        Let $k$ be an imaginary quadratic field containing an $m$-th root of unity;
        let $\bG=\SL_{2}/k$ and let $\tilde\bG$ be the canonical metaplectic extension of $\bG$ by $\mu_{m}$.
        Then for $K_{f}$ sufficiently small there is a non-trivial character $\varepsilon:\mu_{m}\to \C^{\times}$ such that
        the space $H^{2}(Y(K_{f}),\C\otimes\varepsilon)$ is non-zero.
\end{proposition}

\begin{proof}
In the case of $\SL_{2}$, the Kubota symbol has been determined on $\Gamma=\Gamma(m^{2})$.
It is given by (see Proposition 1 in \S3 of \cite{kubota-1969})
\begin{equation}
        \label{kubotasymbol}
        \kappa\begin{pmatrix} a&b \\ c& d\end{pmatrix}
        =
        \begin{cases}
                \left(\frac{c}{d}\right)_{k,m} & \hbox{if $c\ne 0$,}\\
                1 & \hbox{otherwise.}
        \end{cases}
\end{equation}
Here the notation $\left(\frac{c}{d}\right)_{k,m}$ means the $m$-th power residue symbol in the field $k$.
Let $\Gamma_{0}=\ker(\kappa)$.
Then we have a decomposition
\[
        H^{2}(\Gamma_{0},\C)
        =
        \bigoplus_{\eta:\mu_{m}\to \C^{\times}}
        H^{2}(\Gamma,\eta\circ\kappa).
\]
We shall suppose that each of the spaces $H^{2}(\Gamma,\eta\circ\kappa)$
is zero apart from the space where $\eta$ is trivial,
and so we are assuming $H^{2}(\Gamma_{0},\C)=H^{2}(\Gamma,\C)$.

We shall write $Y$ for the arithmetic quotient $\Gamma\backslash X$ and $\tilde Y$
for the $\mu_{m}$-cover $\Gamma_{0}\backslash X$.
We shall also write $\partial Y$ and $\partial \tilde Y$ for the boundaries of the
Borel--Serre compactifications of $Y$ and $\tilde Y$ respectively.
We have a commutative diagram, in which the rows are
exact:
\[
        \begin{matrix}
                H^{2}(Y,\C)
                &\to&
                H^{2}(\partial Y,\C)
                &\to&
                H^{3}_{\cpct}(Y,\C)
                &\to&
                0
                \medskip\\
                \downarrow&&
                \downarrow&&
                \downarrow
                \medskip\\
                H^{2}(\tilde Y,\C)
                &\to&
                H^{2}(\partial \tilde Y,\C)
                &\to&
                H^{3}_{\cpct}(\tilde Y,\C)
                &\to&
                0
        \end{matrix}
\]
We are assuming that the first vertical arrow is an isomorphism.
Since $Y$ and $\tilde Y$ are both connected topological 3-manifolds, it follows that
$H^{3}_{\cpct}(Y,\C)$ and $H^{3}_{\cpct}(\tilde Y,\C)$ are both one-dimensional,
 and the third vertical map is also an isomorphism.
A diagram chase shows that the middle vertical map is surjective.
However, the middle vertical arrow is known to be injective, since the composition
\[
        H^{2}(\partial  Y,\C) \stackrel{\pr^{*}}{\to} H^{2}(\partial \tilde Y,\C) \stackrel{\pr_{*}}{\to} H^{2}(\partial  Y,\C)
\]
is known to be scalar multiplication by the degree of the cover $\partial \tilde Y \to \partial Y$.
We have therefore shown that the map $H^{2}(\partial \tilde Y,\C) \to H^{2}(\partial  Y,\C)$
is bijective.
By examining the map $\partial \tilde Y\to \partial Y$, will show that this is not the case.

Recall that the connected components of $\partial Y$ correspond to $\Gamma$-conjugacy
classes of Borel subgroups $\bB=\bT\bN$ defined over $k$.
For each such Borel subgroup we let $\Gamma_{\bB}=\Gamma \cap \bB(k)$.
As $\Gamma$ is torsion-free, we have $\Gamma_{\bB}\subset \bN(k)$, and the corresponding
boundary component is defined by $\partial Y(\bB) = \Gamma_{\bB}\backslash \bN(\C)$.
Since $\Gamma_{\bB}$ is a lattice in $\bN(\C) \cong \C$, we see that each boundary component is
a $2$-torus.
Hence the dimension of $H^{2}(\partial Y,\C)$ is exactly the number of
 $\Gamma$-conjugacy classes of Borel subgroups.
Each $\Gamma$-conjugacy class of Borel subgroups
is a finite union of $\Gamma_{0}$-conjugacy classes.
However, since we are assuming that $H^{2}(\partial Y,\C)=H^{2}(\partial \tilde Y,\C)$,
we conclude that the $\Gamma$-conjugacy class of each Borel subgroup
is equal to its $\Gamma_{0}$-conjugacy class.

Recall that a Borel subgroup $\bB$ is called \emph{essential} if the restriction of $\kappa$ to
$\Gamma_{\bB}$ is trivial, or equivalently if $\Gamma_{\bB}\subset \Gamma_{0}$.
Not every Borel subgroup is essential; however the standard Borel subgroup of upper triangular matrices
 is clearly essential by \eqref{kubotasymbol}.
Suppose $\bB$ is any essential Borel subgroup.
Since we are assuming that the $\Gamma$- and $\Gamma_{0}$-conjugacy classes of $\bB$
are equal,
it follows that the inclusion $\Gamma_{0}\to \Gamma$ gives us a bijection
\[
        \Gamma_{0}/\Gamma_{\bB} \cong \Gamma/\Gamma_{\bB}.
\]
From this we conclude that $\Gamma=\Gamma_{0}$, which gives us the desired contradiction.
\end{proof}

\begin{remark}
The argument in the proof of Proposition \ref{non-vanish2}
only shows that $H^{2}(Y(K_{f}),\C\otimes\varepsilon)$ contains some non-trivial Eisenstein cohomology classes.
In fact, one can show that there are also metaplectic cusp forms of cohomological type
on $\tilde\SL_{2}$.
This follows by examining the Shimura correspondence for the group $\tilde\GL_{2}/k$ (see \cite{flicker-1980}).
In particular, it is shown that if $\tilde\pi$ is an automorphic representation of $\tilde \GL_{2}$, then there is
a corresponding automorphic representation $\pi$ of $\GL_{2}$.
If $\pi$ is cuspidal, then so is $\tilde \pi$.
The image of the map $\tilde\pi \mapsto \pi$ is also calculated.
In particular if $\pi$ has level $1$, then it has a preimage $\tilde\pi$.
Finally, one can check that if $\pi$ is of cohomological type, then a certain twist of $\tilde\pi$ will be
of cohomological type.
\end{remark}

\section{Background on $p$-adic representation theory}

\subsection{Continuous cohomology}

Throughout this section, we let $\gG$ be a locally compact, totally disconnected group.
By a continuous $\gG$-module, we shall mean an abelian topological group $V$, together
with an action of $\gG$ by endomorphisms of $V$, such that the map $\gG \times V \to V$ is continuous.
Suppose $V$ and $W$ are continuous $\gG$-modules.
We shall write $H^{\bullet}_{\cts}(\gG,V)$ and $\Ext^{\bullet}_{\gG}(V,W)$ for the continuous cohomology
groups (see for example \cite{casselman-wigner}).
We shall sometimes consider continuous $\gG$-modules $V$, which are locally convex topological vector spaces over a field $E$. In this case, the field $E$ will always be a complete discretely valued subfield
of $\C_{p}$.

We begin by recalling a rather technical aspect of continuous cohomology from \cite{casselman-wigner}.
Let $V$ be a continuous representation of $\gG$ and let $C^{n}(\gG,V)$ be the abelian group
of continuous maps $\gG^{n+1}\to V$.
We regard $C^{n}(\gG,V)$ as a topological group with the compact-open topology.
We have an exact sequence of $\gG$-modules
\begin{equation}
        \label{resolution}
        0 \to V \to C^{0}(\gG,V) \to C^{1}(\gG,V) \to \cdots .
\end{equation}
Recall that $H^{\bullet}_{\cts}(\gG,V)$
 is the cohomology of the cochain complex $C^{n}(\gG,V)^{\gG}$.

\begin{definition} (see \S1 of \cite{casselman-wigner})
        The group \[Z^{n}(\gG,V)=\ker (C^{n}(\gG,V)\to C^{n+1}(\gG,V))\]
        is given the subspace topology;
        the group \[B^{n}(\gG,V)=\im (C^{n-1}(\gG,V)\to C^{n}(\gG,V))\] 
        is given the quotient topology
        as $C^{n-1}(\gG,V)/Z^{n-1}(\gG,V)$.
        Since \eqref{resolution} is exact, the groups $B^{n}(\gG,V)$ and $Z^{n}(\gG,V)$
        are identical for $n>0$, and the identity map gives a continuous bijective
        homomorphism $B^{n}(\gG,V) \to Z^{n}(\gG,V)$.
        We say that the cohomology $H^{\bullet}_{\cts}(\gG,V)$ is \emph{strongly Hausdorff}
        if the maps $B^{n}(\gG,V) \to Z^{n}(\gG,V)$ are open, i.e. if the two topologies are the same.
\end{definition}


\begin{lemma}
        Suppose $\gG$ is a union of countably many compact subsets
        and $V$ is a Fr\'echet space over $E$ with a continuous action of $\gG$.
        Then the cohomology groups $H^{\bullet}_{\cts}(\gG,V)$ are strongly Hausdorff.
\end{lemma}

\begin{proof}
        The spaces $C^{n}(\gG,V)$ are Fr\'echet spaces.
        Hence the closed subspaces $Z^{n}(\gG,V)$ are also Fr\'echet spaces,
        and the quotient spaces $B^{n+1}=C^{n}/Z^{n}$ are Fr\'echet spaces.
        The map $B^{n}(\gG,V)\to Z^{n}(\gG,V)$ is a continuous linear bijection of Fr\'echet spaces.
        By the open mapping theorem \cite[Proposition 8.6]{schneider-nfa}, this map is an isomorphism of topological vector spaces.
\end{proof}

\begin{theorem}[Standard facts about continuous
cohomology]\label{thm:standardfacts}
        Let $\gG$ be a locally compact, totally disconnected topological group
        and $\gH$ a closed subgroup of $\gG$.
        Let $E$ be a field as described above.
        \begin{enumerate}
                \item
                The vector space $\cC(\gG,E)$ of continuous functions from $\gG$ to $E$
                is continuously injective as a module over $\gH$. In particular
                $H^{r}_{\cts}(\gH,\cC(\gG,E))$ is zero for $r>0$.
                \item
                Suppose $\gH$ is normal in $\gG$ and $V$ is a continuous $\gG$-module.
                If the groups $H^{\bullet}_{\cts}(\gH,V)$ are strongly Hausdorff then
                there is a natural continuous action of $\gG/\gH$ on $H^{\bullet}_{\cts}(\gH,V)$,
                 and there is a spectral sequence $E_{2}^{r,s}=H^{r}_{\cts}(\gG/\gH,H^{s}_{\cts}(\gH,V))$
                which converges to $H^{r+s}_{\cts}(\gG,V)$.
                \item
                Suppose $\gG$ is a profinite group.
                Then for each $r>0$ there is a short exact sequence
                \[
                        0 \to \sideset{}{^{(1)}}\varprojlim_t
 H^{r-1}_{\cts}(\gG,\Z/p^{t})
                        \to H^{r}_{\cts}(\gG,\Z_{p})
                        \to
                        \limp{t} H^{r}_{\cts}(\gG,\Z/p^{t})
                        \to 0.
                \]
                The notation $\sideset{}{^{(1)}}\varprojlim$ means the first derived functor of the projective
                limit functor (see for example \cite[section 3.5]{weibel}).
        \end{enumerate}
\end{theorem}

\begin{proof}
        Part (1) is a special case of Proposition 4(a) of \cite{casselman-wigner}.
        Part (2) is proposition 5 of \cite{casselman-wigner}.
        Part (3) is a special case of Theorem 2.3.4 of \cite{n-s-w}.
\end{proof}

\subsection{Some functional analysis}

We shall consider continuous representations of a topological group on
locally convex topological vector spaces over a coefficient field $E$
containing $\Qp$. We again suppose that $E$ is a complete discretely valued
subfield of $\mathbb{C}_p$, so in particular $E$ is spherically complete
\cite[Lemma 1.6]{schneider-nfa}.

For two topological vector spaces $V,W$, we shall write $\cL(V,W)$ for the vector space of
continuous linear maps from $V$ to $W$.
We shall always regard $\cL(V,W)$ as a topological vector space with the strong topology,
and sometimes write $\cL_{b}(V,W)$ to emphasise this.
The notation $V'$ will mean the strong dual of $V$.

Recall that a Fr\'echet space $V$ may be written as the projective limit of a sequence
\[
        V_{1} \leftarrow V_{2} \leftarrow \cdots,
\]
where each $V_{i}$ is a Banach space and each transition map is surjective.
If each transition map is nuclear, then $V$ is called a \emph{nuclear Fr\'echet space}.
A topological vector space $V$ over $E$ is said to be
\emph{of compact type} if it is the locally convex inductive limit of a sequence
\[
        V_{1} \to V_{2}\to V_{3} \to \cdots,
\]
where each $V_{i}$ is a Banach space over $E$, and each transition map is
compact and injective. Compact type spaces are Hausdorff, complete,
bornological, reflexive (and hence barrelled).
The strong dual of a compact type
space is a nuclear Fr\'echet space and vice versa.
More precisely, the functor which takes a
compact type space to its strong dual is an antiequivalence of categories,
between the category of compact type spaces and the category of nuclear Fr\'echet
spaces.

Suppose $V$ and $W$ are Fr\'echet spaces. There is a canonical topology on
$V\otimes W$, in which the continuous bilinear maps $V\times W \to X$ correspond
to the continuous linear maps $V\otimes W \to X$. This topology is not, in
general, Hausdorff. We shall write $V \mathbin{\hat\otimes} W$ for the Hausdorff
completion of $V\otimes W$ with respect to this topology.

\subsection{Continuous admissible representations}

In this section, we shall suppose that we have a connected reductive group $\cG$
defined over $\Qp$, and we write $\scG$ for
the group of $\Qp$-valued points. We shall suppose also that we have a
topological central extension
\[
        1 \to \mu\to \tilde \scG \stackrel{\pr}{\to} \scG \to 1,
\]
where $\mu$ is a finite abelian group.

Let $\scK$ be a compact open subgroup of $\tilde \scG$ and let
$\cC(\scK)$ be the vector space of continuous functions $f:\scK\to E$.
The supremum norm on functions makes $\cC(\scK)$ into a Banach space over
$E$. We shall write $\cD(\scK)$ for its strong dual. The space $\cD(\scK)$
is naturally a Banach algebra over $E$, with multiplication given by convolution of distributions.
This algebra is known to be Noetherian \cite[Theorem 6.2.8]{emerton-memoir}.
If $V$ is a continuous representation of $\tilde \scG$, and $V$ is also a Banach space,
then there is a natural action of $\cD(\scK)$ on the dual space $V'$ (see \cite[Proposition
5.1.7]{emerton-memoir}).

\begin{definition}[{\cite[Proposition-Definition 6.2.3]{emerton-memoir}}]
        \label{cts-admiss}
        Let $V$ be a continuous representation of $\tilde \scG$ on a Banach space.
        We call this representation \emph{admissible continuous} if the dual space $V'$
        is finitely generated as a $\cD(\scK)$-module.
        This condition does not depend on the choice of compact open subgroup $\scK$.
\end{definition}

If the coefficient field $E$ is a finite extension of $\Q _p$, then the
above definition is equivalent to the definition of an admissible continuous
representation given in \cite[\S 3]{ST-banach}.


\subsection{Locally analytic representations}

We shall write $\gg$ for the Lie algebra of $\cG$ over $\Qp$.
By a \emph{Lie sublattice} $\gh$ in $\gg$, we shall mean a finitely generated $\Z_{p}$-submodule, which spans
$\gg$ over $\Qp$, and which is closed under the Lie bracket operation.
Such a sublattice defines a norm on $\gg$, with respect to which $\gh$ is the unit ball.
Hence there is an affinoid $\bbH$, such that $\gh = \bbH(\Qp)$.
If $\gh$ is sufficiently small, then the Baker--Campbell--Hausdorff formula converges
on $\bbH$, and gives $\bbH$ the structure of a rigid analytic group. Furthermore the exponential
map converges on $\bbH$, and gives a bijection $\exp:\bbH(\Qp) \to \scH$ for some compact open subgroup
$\scH$ of $\scG$.
A subgroup $\scH$ which arises in this way is called a \emph{good analytic open subgroup of $\scG$}.
Recall that there is a compact open subgroup $\scK$ of $\scG$, which lifts to
a subgroup $\hat \scK$ of $\tilde \scG$. We shall fix such a subgroup, together with its
lift.

\begin{definition}
        By a \emph{good analytic open subgroup} of $\tilde \scG$, we shall mean
        a compact open subgroup $\hat \scH$ of $\hat \scK$, such that the image
        of $\hat \scH$ in $\scG$ is a good analytic open subgroup of $\scG$.
\end{definition}

Given a good analytic open subgroup $\scH$ of $\tilde \scG$ and a Banach space $V$,
we write $\cC^{\an}(\bbH,V)$ for the vector space of $V$-valued functions in $\scH$,
which are given by a power series which converges on the whole of $\bbH$.
We regard $\cC^{\an}(\bbH,V)$ as a Banach space in which the topology is given by the
supremum norm on $\bbH$.
More generally, if $V$ is a Hausdorff locally compact topological vector space, then we define
\[
        \cC^{\an}(\bbH,V)
        =
        \limd{W\to V} \cC^{\an}(\bbH,W),
\]
where $W\to V$ runs through Banach spaces which map continuously and injectively into $V$.

Suppose that $V$ is also a continuous representation of $\tilde \scG$.
We call a vector $v\in V$ $\scH$-analytic if the orbit map $\scH \to V$ given by $h \mapsto hv$
is represented by an element of $\cC^{\an}(\bbH,V)$.
We shall write $V^{\scH-\an}$ for the subspace of $\scH$-analytic vectors in $V$.
The topology on $V^{\scH-\an}$ is defined to be that given by the supremum norm on $\cC^{\an}(\bbH,V)$.

A vector $v\in V$ is said to be \emph{locally analytic} if it is $\scH$-analytic for
a suitable good analytic open subgroup $\scH$.
The subspace of locally analytic vectors in $V$ will be written $V_{\la}$.
We have an isomorphism of vector spaces:
\[
        V_{\la}
        =
        \limd{\scH}
        V^{\scH-\an}.
\]
We shall regard $V_{\la}$ as a topological vector space with the direct limit topology.
If $\scH_{1}$ is a proper subgroup of $\scH_{2}$ then the map
$V^{\scH_{2}-\an} \to V^{\scH_{1}-\an}$ is compact, and so $V_{\la}$ is a compact type space.

There is a natural continuous map $V_{\la} \to V$.
We call $V$ a \emph{locally analytic} representation if this map is an isomorphism of
topological vector spaces.

\subsection{Fr\'echet--Stein algebras}

Let $A$ be a locally convex topological $E$-algebra.
A Fr\'echet--Stein structure \cite{ST-admissible} on $A$ is an isomorphism of locally convex topological $E$-algebras,
\[
        A = \limp{n} A_{n},
\]
such that
\begin{enumerate}
        \item
        each $A_{n}$ is a left-Noetherian Banach algebra;
        \item
        each map $A_{n+1}\to A_{n}$ is a continuous homomorphism,
        and is right-flat.
\end{enumerate}
An algebra $A$ with such a structure is called a Fr\'echet--Stein algebra.
Suppose $A=\limp{n} A_{n}$ is a Fr\'echet--Stein algebra and $M$ is an $A$-module.
We say that $M$ is \emph{coadmissible} if there is an isomorphism of $A$-modules
\[
        M = \limp{n} M_{n},
\]
such that
\begin{enumerate}
        \item
        each $M_{n}$ is a finitely generated locally convex topological $A_{n}$-module;
        \item
        each map $M_{n+1}\to M_{n}$ of $A_{n+1}$-modules induces an isomorphism
        $M_{n+1}\otimes_{A} A_{n}\to M_{n}$ of $A_{n}$-modules.
\end{enumerate}
If $M=\limp{n}M_{n}$ is coadmissible, then it automatically follows that $M_{n}=M \otimes_{A} A_{n}$. The category of coadmissible modules over a Fr\'echet--Stein algebra
has many of the same good properties as the category of finitely--generated
modules over a Noetherian Banach algebra (which is a special case); in
particular, it is an abelian category.

\subsection{Locally analytic distributions}

Let $\scK$ be a compact open subgroup of $\tilde \scG$.
Fix a good analytic open subgroup $\scH_{1}$ of $\scK$.
A function $\scK\to E$ is said to be $\scH_{1}$-\emph{analytic} if its restriction to every $\scH_{1}$-coset
 can be written as a power series expansion on the corresponding Lie sublattice in $\gg$.
The $\scH_{1}$-analytic functions on $\scK$ form a Banach space with respect to the supremum norm
on the functions.
We shall call this space $\cC^{\scH_{1}-\an}(\scK)$.

If $\scH_{2}\subset \scH_{1}$ is a proper subgroup, and is also a good analytic open subgroup,
then every $\scH_{1}$-analytic function is $\scH_{2}$-analytic, and so we have an inclusion
\[
        \cC^{\scH_{1}-\an}(\scK) \hookrightarrow \cC^{\scH_{2}-\an}(\scK).
\]

A function on $\scK$ is said to be \emph{locally analytic} if there
exists a good analytic open subgroup $\scH_{1}$, such that the function is
$\scH_{1}$-analytic. We shall write $\cC^{\locan}(\scK)$ for the space of such
functions. We clearly have
\[
        \cC^{\locan}(\scK)
        =
        \limd{n}
        \cC^{\scH_{n}-\an}(\scK),
\]
where $\scH_{n}$ is a basis of neighbourhoods of the identity in $\tilde \scG$
consisting of good analytic open subgroups.
By construction, $\cC^{\locan}(\scK)$ has compact type.
We shall write $\cD^{\locan}(\scK)$ for its strong dual, which is therefore
a nuclear Fr\'echet space.
Furthermore, there is a convolution multiplication on
$\cD^{\locan}(\scK)$ induced by the group law on $\scK$.
We can write $\cD^{\locan}(\scK)$ as a projective limit of Banach algebras:
\[
        \cD^{\locan}(\scK)
        =
        \limp{n} \cD^{\scH_{n}-\an}(\scK),
        \qquad
        \cD^{\scH_{n}-\an}(\scK)
        =
        \cC^{\scH_{n}-\an}(\scK)'.
\]
This gives $\cD^{\locan}(\scK)$ the structure of a Fr\'echet--Stein algebra
\cite[Theorem 5.1]{ST-admissible}.

\subsection{Admissible locally analytic representations}

We now recall the definition of an admissible locally analytic
representation, introduced in \cite{ST-admissible}.
If $V$ is a locally analytic representation of $\tilde \scG$, and $\scK$ is a compact open subgroup of $\tilde
\scG$, then the action of $\scK$ on $V'$ extends to a continuous
$\cD^{\la}(\scK)$-module structure (see \cite[Proposition
3.2]{ST-distributions} or \cite[Proposition 5.1.9(ii)]{emerton-memoir}). We say
$V$ is \emph{admissible locally analytic} if $V'$ is coadmissible for one, or
equivalently every, open compact $\scK$.
On the other hand, $V$ is said to be strongly admissible if $V'$ is finitely generated over $\cD^{\la}(\scK)$.
Since every finitely generated module is coadmissible, it follows that every strongly admissible locally analytic
representation is an admissible locally analytic representation.

\begin{theorem}
        Let $V$ be an admissible continuous representation of $\tilde \scG$.
        Then the subspace $V_{\la}$ of locally analytic vectors in $V$ has the structure of
        a locally analytic representation of $\tilde \scG$.
        Furthermore, $V_{\la}$ is strongly admissible.
\end{theorem}

\begin{proof}
This is shown in \cite[Theorem
7.1]{ST-admissible} under the mild additional hypothesis that $E$ is a finite
extension of $\Q _p$, and for general complete discretely valued $E$ in
\cite[6.2.4]{emerton-memoir} (which is stated for locally analytic groups).
\end{proof}

If $V$ and $W$ are locally analytic representations of $\tilde \scG$, then they have an action of
the Lie algebra $\gg$ of $\cG$ over $\Qp$, and we shall write $H^{\bullet}_{\Lie}(\gg,V)$ and
$\Ext^{\bullet}_{\gg}(V,W)$ for the Lie algebra cohomology.
There is a smooth action of $\tilde \scG$ on $\Ext^{\bullet}_{\gg}(V,W)$.

\begin{theorem}[Emerton]\label{thm:liecohomology}
        Let $V$ be an admissible continuous representation of $\scK$ and let $W$ be a finite dimensional
        algebraic representation of $\scK$.
        Then there are canonical isomorphisms
        \[
                \Ext^{\bullet}_{\scK}(W,V)
                \cong
                \Ext^{\bullet}_{\gg}(W,V_{\la})^{\scK},
                \quad\quad
                H^{\bullet}_{\cts}(\scK,V)
                \cong
                H^{\bullet}_{\Lie}(\gg,V_{\la})^{\scK}.
        \]
        In particular we have
        \[
                \Ext^{\bullet}_{\gg}(W,V_{\la})
                \cong
                \limd{\scU}\Ext^{\bullet}_{\scU}(W,V),
                \quad\quad
                H^{\bullet}_{\Lie}(\gg,V_{\la})
                \cong
                \limd{\scU}H^{\bullet}_{\cts}(\scU,V),
        \]
        where the limits are taken over subgroups $\scU$ of finite index in $\scK$.
\end{theorem}

\begin{proof}
 The formulae of the left column follow from those of the right column applied
to $V \otimes W'$, so it suffices to prove the latter. These follow from
\cite[Prop. 1.1.12(ii) and Theorem 1.1.13]{emerton-interpolation}.
\end{proof}


\subsection{Essentially admissible locally analytic representations}

The definition of essentially admissible locally analytic representations,
introduced in \S 6.4 of \cite{emerton-memoir} is slightly more involved. The
group $\tilde \scG$ is locally $\Qp$-analytic, and its centre $Z = Z_{\tilde
\scG}$ is topologically finitely generated. Following the construction \emph{loc.cit.}, we
may construct a rigid
space $\widehat Z$ over $E$, together with a ``universal character'' $Z
\hookrightarrow \cC^{\an}(\widehat Z)$. This space parametrises the locally $\Qp$-analytic 
characters of $Z$, in the sense that if $E' / E$ is any finite extension, there is a canonical 
bijection between the $E'$-points of $\widehat Z$ and the $E'$-valued characters of $Z$, where 
a point $x \in \widehat Z(E')$ corresponds to the composition of the universal character with 
the evaluation map at $x$.

If $V$ is a locally analytic representation of $\tilde \scG$, then 
the dual space $V'$ has an action of $\cD^{\la}(\scK)$.
Suppose also that the $Z_{\tilde \scG}$-action on $V$ extends to a separately continuous
action
\[
        \cC^{\an}(\widehat Z) \times V \to V.
\]
Then by \cite[Proposition 6.4.7]{emerton-memoir}, $V'$ is also a topological
$\cC^{\an}(\widehat Z)$-module.
The actions of $\cC^{\an}(\widehat Z)$ and $\cD^{\la}(\scK)$
on $V'$ commute,
and so $V'$ has an action of $\cC^{\an}(\widehat Z) \mathbin{\hat\otimes} \cD^{\la}(\scK)$,
which is a Fr\'echet--Stein algebra.

\begin{definition}
        \label{essen-admiss}
        Let $V$ be a locally analytic representation of $\tilde \scG$. We say that $V$
        is \emph{essentially admissible} if
        it satisfies the following two conditions:
        \begin{enumerate}
                \item
                the $Z$-action on $V$ extends to a separately continuous
                action of $\cC^{\an}(\widehat Z)$,
                \item
                for one, and hence for every compact open subgroup $\scK$ of $\tilde \scG$,
                the dual space $V'$ is co-admissible as a module over the Fr\'echet--Stein
                algebra $\cC^{\an}(\widehat Z) \mathbin{\hat\otimes} \cD^{\la}(\scK)$.
        \end{enumerate}
\end{definition}

\begin{theorem}
        Every admissible locally analytic representation
         of $\tilde \scG$ is an essentially admissible locally analytic representation.
\end{theorem}

\begin{proof}
        This is \cite[Proposition 6.4.10]{emerton-memoir}, and is stated there for arbitrary locally analytic groups.
\end{proof}

Suppose that the centre $Z$ has finite index in $\tilde\scG$, and let $V$ be an essentially
admissible locally analytic representation of $\tilde \scG$.
Then $V$ is by restriction an essentially admissible representation of $Z$.
For any affinoid $\mathbb{U} \subset \widehat Z$, the module $V' \otimes_{\cC^{\an}(\widehat Z)} \cC^{\an}(\mathbb{U})$ is
finitely generated over $\cC^{\an}(\mathbb{U})$, and so we may regard $V'$ as a coherent sheaf on $\widehat Z$.
The functor which takes $V$ to $V'$ is an anti-equivalence of categories between the category of
essentially admissible locally analytic representations of $Z$ and the category
of coherent sheaves on $\widehat Z$ (see the discussion following Proposition 6.4.10
in \cite{emerton-memoir}).

\section{Completed cohomology of metaplectic groups}

In this section, we adapt some definitions and results of Emerton
\cite{emerton-interpolation} to the metaplectic case. Suppose again that we have
a connected reductive group $\bG$ defined over a number field $k$, and that we
have a type 1 metaplectic cover of $\bG$ by $\mu$.

Given a representation $W$ of $\tilde K_{\gp}$, we may define a corresponding
local system on
 $Y(K^{\gp}K_{\gp})$ as follows:
\[
        \cV_{\gp}(W)
        =
        \left( (\hat\bG(k) \backslash \tilde \bG(\A) / \hat K^{\gp} \hat K_{\infty}^{\circ})\times W \right) / \tilde K_{\gp}.
\]
Suppose as before that our coefficient field $E$ is an extension of $\Qp$.
Again, let $W$ be an algebraic representation of $\cG$ over $E$ and let $\varepsilon:\mu\to E^{\times}$
be an injective character.
We then have an action of $G_{\gp}$ on $W$. This gives rise to an action of
$\tilde K_{\gp} =\hat K_{\gp} \oplus \mu$, in which $\hat K_{\gp}$ acts through its
isomorphism with $K_{\gp}$ and $\mu$ acts by scalar multiplication by $\varepsilon$.
We shall call this representation $W\otimes\varepsilon$.
We therefore have a local system $\cV_{\gp}(W\otimes\varepsilon)$ on $Y(K^{\gp}K_{\gp})$.

As in Lemma 2.2.4 of Emerton \cite{emerton-interpolation}, we note that
$\cV_{\gp}(W\otimes\varepsilon)$ is canonically isomorphic to the local system $\cV_{W\otimes\varepsilon}$
defined in the introduction.
In particular, these two local systems have the same cohomology groups.
It is important to see these cohomology groups from both points of view: when regarding them as the cohomology of $\cV_{W\otimes\varepsilon}$,
it is clear that these groups are defined over any field of definition of $W$.
In particular, it follows that the eigenvalues of the Hecke operators on these spaces are algebraic.
When regarding them as cohomology groups of $\cV_{\gp}(W\otimes\varepsilon)$,
we shall see that we are able to $p$-adically interpolate the systems of eigenvalues.

\subsection{The representations $\tH^{s}$}
\label{sect:htilde}

Let $\cC(\tilde K_{\gp})$ be the vector space
of continuous functions $f: \tilde K_{\gp}\to E$.
The vector space $\cC(\tilde K_{\gp})$ is a continuous representation of $\tilde K_{\gp}\times \tilde K_{\gp}$,
where the first $\tilde K_{\gp}$ acts on function by left-translation and the second by right-translation.
Using one of these $\tilde K_{\gp}$ actions, we can define a local system $\cV(\tilde K_{\gp})$ of $\tilde K_{\gp}$-modules on $ Y(\hat K^{\gp} \hat K_{\gp})$ by
\[
        \cV(\tilde K_{\gp})
        =
        \left(  (\hat\bG(k)  \backslash  \tbG(\A) / \hat K_{\infty}^{\circ} \hat K^{\gp}) \times\cC(\tilde K_{\gp})\right)/
        \tilde K_{\gp},
\]
and we shall be interested in the cohomology groups of this local system:
\[
        \tH^{\bullet}(\hat K^{\gp},E)
        =
        H^{\bullet}(Y(\hat K^{\gp}\hat K_{\gp}),\cV(\tilde K_{\gp})).
\]
The vector spaces $\tH^{\bullet}(\hat K^{\gp},E)$ have an action of $\tilde K_{\gp}$, together with
a commuting actions of $\cH^{\gp}$.
To avoid sign errors, we state now that we have used the right-translation action of $\tilde K_{\gp}$
to form the local system $\cV(\tilde K_{\gp})$;
the action of $\tilde K_{\gp}$ on $\tH^{\bullet}(\hat K^{\gp},E)$ is given by the left-translation action.

Note that the vector space $\cC(\tilde K_{\gp})$ decomposes as a direct sum of $\mu$-isotypic subspaces:
\[
        \cC(\tilde K_{\gp})
        =
        \bigoplus_{\eta : \mu\to E^{\times}}
        \cC(\tilde K_{\gp})^{\eta},
\]
where $\cC(\tilde K_{\gp})^{\eta}$ consists of continuous functions $f:\tilde K_{\gp}\to E$,
such that $f(\zeta x)=\eta(\zeta)\cdot f(x)$ for all $\zeta\in\mu$.
As a consequence, we have a similar decomposition
\[
        \tH^{\bullet}(\hat K^{\gp},E)
        =
        \bigoplus_{\eta : \mu \to E^{\times}}
        \tH^{\bullet}_{\eta}(\hat K^{\gp},E).
\]
Note that as the left and right translation actions of $\mu$ differ by a sign, we have the following:

\begin{lemma}
        \label{epsilon-eigenspace}
        The subspace $\tH^{\bullet}_{\eta}(\hat K^{\gp},E)$ of $\tH^{\bullet}(\hat K^{\gp},E)$
        is the $\eta'$-eigenspace for the action of $\mu$.
\end{lemma}

\begin{lemma}
        As representations of $\tilde K_{\gp}$, the spaces $\tH^{\bullet}(\hat K_{\gp},E)$ are admissible continuous representations.
\end{lemma}

\begin{proof}
        Recall that $Y(K^{\gp}K_{\gp})$ is homotopic to a finite simplicial complex $Y$.
        Let $Y(d)$ be the set of simplices of $Y$ of dimension $d$.
        Then $\tH^{\bullet}(\tilde K_{\gp},E)$ is the cohomology of the chain complex
        \begin{equation}\label{eq:complex}
                0 \to \cC(\tilde K_{\gp})^{Y(0)} \to \cC(\tilde K_{\gp})^{Y(1)} \to \cC(\tilde K_{\gp})^{Y(2)} \to \cdots .
        \end{equation}
        Hence each cohomology group is a subquotient of finitely many copies of $\cC(\tilde K_{\gp})$.
\end{proof}

\begin{theorem}
        \label{thm4.3}
        There is a canonical isomorphism:
        \[
                \tH^{\bullet}(\hat K^{\gp}, E)
                \cong
                \left(\limp{n}\limd{K_{\gp}} H^{\bullet}(\tilde Y(\hat K^{\gp}\hat K_{\gp}),\cO_E / p^{n}) \right)
                \otimes_{\cO_E} E.
        \]
\end{theorem}
\begin{proof}
        This is a special case of \cite[Theorems 2.5 and 2.10]{hill10}.
\end{proof}

\begin{corollary}
        The action of $\tilde K_{\gp}$ on $\tH^{\bullet}(\hat K^{\gp},E)$ extends to
        a canonical action of $\tilde G_{\gp}$.
\end{corollary}

\begin{proof}
        This is immediate from the previous theorem, since
        $\tilde G_{\gp}$ already acts on the space
        \[
                \limd{K_{\gp}} H^{\bullet}(\tilde Y(\hat K^{\gp}\hat K_{\gp}),\cO_E/p^{n}).
        \]
\end{proof}

\begin{corollary}
        The group $\tH^{\bullet}(\hat K^{\gp},E)$ is independent (up to a canonical isomorphism)
        on the choice of $K_{\gp}$.
\end{corollary}

\begin{proof}
        This is immediate from Theorem \ref{thm4.3}.
\end{proof}

\subsection{Some spectral sequences}

\begin{theorem}
        Let $W$ be any continuous representation of $\tilde K_{\gp}$ over $E$,
        and let $W'$ be the contragredient representation on the continuous dual space.
        There is a spectral sequence
        \[
                \Ext^{r}_{\tilde K_{\gp}}(W , \tH^{s}(\hat K^{\gp},E))
                \Rightarrow
                H^{r+s}(Y( K^{\gp} K_{\gp}) , \cV_{\gp}(W')).
        \]
\end{theorem}

\begin{proof}
        This is a special case of Theorem 3.5 of \cite{hill10}.
        We note that for algebraic $W$, Emerton's original proof in \cite{emerton-interpolation} extends to the metaplectic case without modification.
\end{proof}

In what follows we let $\cG=\Rest^{k_{\gp}}_{\Qp}(\bG \times_{k} k_{\gp})$.
Thus $\cG$ is a reductive group over $\Qp$ and we may canonically identify
$\cG(\Qp)$ with $\bG(k_{\gp})$.
Suppose that $W$ is an algebraic representation of $\cG$ over $E$.
For a character $\varepsilon : \mu\to E^{\times}$, we let $W\otimes \varepsilon$ be the
representation of $\tilde K_{\gp}$ in which $\hat K_{\gp}$ acts through $K_{\gp}$ and $\mu$
acts by $\varepsilon$.

\begin{corollary}
        There is a spectral sequence
        \[
                \Ext^{r}_{\hat K_{\gp}}(W' , \tH^{s}_{\varepsilon}(\hat K^{\gp},E))
                \Rightarrow
                H^{r+s}(Y(K^{\gp} K_{\gp}), \cV_{W\otimes\varepsilon}).
        \]
\end{corollary}

\begin{proof}
        We shall apply the previous theorem to the representation $(W\otimes\varepsilon)'$.
        To simplify notation, we shall write $\tH^{s}$ in place of $\tH^{s}(\hat K^{\gp},E)$.
        By the theorem, we have a spectral sequence whose $E_{2}^{r,s}$ term is
        $\Ext^{r}_{\tilde K_{\gp}}(W'\otimes\varepsilon' , \tH^{s})$.
        The corollary is proved by the following calculation:
        \begin{align*}
                \Ext^{r}_{\tilde K_{\gp}}(W'\otimes\varepsilon' , \tH^{s})
                &=
                \Ext^{r}_{\hat K_{\gp}}(W', \Hom_{\mu}( \varepsilon' , \tH^{s}))\\
                &=
                \Ext^{r}_{\hat K_{\gp}}(W' , \tH^{s}_{\varepsilon}).
        \end{align*}
        In the first line above we have used the Hochschild--Serre spectral sequence, which degenerates
        because $\mu$ is finite and $E$ has characteristic zero.
        The second line is immediate from Lemma \ref{epsilon-eigenspace}.
\end{proof}

\begin{corollary}
        \label{spectral-sequence}
        There is a spectral sequence
        \[
                \Ext^{r}_{\gg}(W' , \tH^{s}_{\varepsilon}(\hat
K^{\gp},E)_{\la})
                \Rightarrow
                H^{r+s}_{\cl, \varepsilon}(\hat K^{\gp}, W).
        \]
\end{corollary}

\begin{proof}
This follows by taking the direct limit of the formula of the previous
corollary, over levels $\hat K_{\gp}$, and applying Theorem
\ref{thm:liecohomology}.
\end{proof}

\begin{definition}
        We shall say that the triple $(\tbG, \hat K^{\gp},\varepsilon)$ satisfies the
        \emph{edge map criterion} in dimension $n$
        if for every finite dimensional algebraic representation $W$ of $\cG$,
        the edge map in the spectral sequence above gives an isomorphism
        \[
                \Hom_{\gg}(W', \tH^{n}_{\varepsilon}(\hat K^{\gp}, E)_{\la})
                \cong
                H^{n}_{\cl,\varepsilon}(\hat K^{\gp},W).
        \]
\end{definition}

\subsection{Calculation of certain spaces $\tH^{n}_{\varepsilon}$}

In this section we let $\bG$ be absolutely simple, simply connected, and of positive real rank.
We recall if $k$ contains a primitive $m$-th root of unity, then there is a canonical
metaplectic extension of $\bG$ by $\mu_{m}$, and this extension is the universal metaplectic
extension if $\mu_{m}$ is the group of all roots of unity in $k$.
We shall assume that $\tbG$ is one of these canonical extensions.

Recall that such a group $\bG$ satisfies strong approximation,
and so we may identify the arithmetic quotient $Y(K^{\gp}K_{\gp})$
with $\Gamma \backslash X$, where $X$ is the symmetric space $G_{\infty}/K_{\infty}$,
and $\Gamma$ is the congruence subgroup of level $K^{\gp}K_{\gp}$.
The local system $\cV(K_{\gp})$ may by identified with $\Gamma \backslash (X \times \cC(\tilde K_{\gp}))$,
where the action of $\Gamma$ on $\cC(\tilde K_{\gp})$ is by right-translation.
As a consequence, we have
\[
        \tH^{\bullet}(\hat K^{\gp},E)
        =
        H^{\bullet}_{\group}(\Gamma, \cC(\tilde K_{\gp})).
\]
Here we are regarding $\Gamma$ as a subgroup of $\tilde K_{\gp}$ through the
maps
\[
        \Gamma  \cong \{g\in \hat\bG(k) : \pr(g)\in \Gamma\} \hookrightarrow \tilde K_{f} = \tilde K_{\gp}\times  \hat K^{\gp} \to \tilde K_{\gp}.
\]

\begin{theorem}
        Let $\tbG$ be as described above.
        If $\varepsilon$ is non-trivial then $\tH^{0}_{\varepsilon}(\hat K^{\gp},E)=0$.
        If $\varepsilon$ is trivial then $\tH^{0}_{\varepsilon}(\hat K^{\gp},E)\cong E$.
\end{theorem}

\begin{proof}
By strong approximation, the arithmetic quotient $Y(K^{\gp}K_{\gp})$ is connected,
and so $\tH^{0}(\hat K^{\gp})$ is the space of elements of $\cC(\tilde K_{\gp})$,
which are right-$\tilde K_{\gp}$-invariant.
This is simply the space of constant functions.
The action of $\mu$ by left translations is trivial on the constant functions.
This proves the result.
\end{proof}

%

By the strong approximation theorem, the closure of $\Gamma$ in $\bG(\A_{f})$ is $K_{f}$.
We shall write $\overline K_{f}$ for the profinite completion of $\Gamma$.
There is a canonical surjective homomorphism
$\overline K_{f} \to K_{f}$.
The \emph{congruence kernel} $\Cong(\bG)$ is defined to be the kernel of this map.
We therefore have a short exact sequence of profinite groups:
\[
        1 \to \Cong(\bG) \to \overline K_{f} \to K_{f} \to 1.
\]
The congruence kernel measures the extent to which $\Gamma$ has non-congruence
subgroups of finite index.
The Kubota symbol gives a map $\Cong(\bG)\to\mu_{m}$, and we let $\Cong(\bG)_{0}$ be the kernel of this map.
It is conjectured (and in many cases proved) that when $\mu_{m}$ is the group of all roots of unity in $k$,
the Kubota symbol is an isomorphism whenever $\bG$ has real rank at least $2$.
When $\bG$ has real rank $1$, it is conjectured that $\Cong(\bG)$ is infinite.


\begin{theorem}
        Let $\tbG$ be as described above.
        There is a canonical isomorphism
        \[
                \tH^{1}(\hat K^{\gp},E)
                =
                \Hom_{\cts,\hat K^{\gp}}(\Cong_{0},E).
        \]
\end{theorem}

\begin{proof}
        Recall that $K_{f}$ is a compact open subgroup of $\bG(\A_{f})$, which lifts to a subgroup
        $\hat K_{f}$ of $\tbG(\A_{f})$.
        We shall choose such a $K_{f}$ of the form $K^{\gp}K_{\gp}$.
        As before, we write $\tilde K_{f}$ for the preimage of $K_{f}$ in $\tbG(\A_{f})$.
        Furthermore let $\bar K_{f}$ be the profinite completion of $\Gamma(K_{f})$.
        The congruence kernel is defined to be the kernel of the canonical map $\bar K_{f}\to K_{f}$.
        This map factors through $\tilde K_{f}$, and the kernel of the map $\bar K_{f}\to \tilde K_{f}$
        is $\Cong_{0}$.
        We therefore have an extension of groups:
        \[
                1 \to \Cong_{0} \to \bar K_{f} \to \tilde K_{f} \to 1.
        \]
        Let $\cC(\bar K_{f})$ be the vector space of continuous functions from $\bar K_{f}$ to $E$.
        We shall regard $\cC(\bar K_{f})$ as a $\Gamma\times \Cong_{0}$-module.
        By Theorem \ref{thm:standardfacts} there are two spectral sequences:
        \begin{align*}
                H^{r}(\Gamma,H^{s}_{\cts}(\Cong_{0},\cC(\bar K_{f})))
                &\Rightarrow
                H^{r+s}_{\cts}(\Gamma\times \Cong_{0},\cC(\bar K_{f})),\\
                H^{r}_{\cts}(\Cong_{0},H^{s}(\Gamma,\cC(\bar K_{f})))
                &\Rightarrow
                H^{r+s}_{\cts}(\Gamma\times \Cong_{0},\cC(\bar K_{f})).
        \end{align*}
        (Here we are regarding $\Gamma \times \Cong_{0}$ as a topological group, in which
        the profinite group $\Cong_{0}$ is an open subgroup.)
        Again by Theorem \ref{thm:standardfacts}, the first of these spectral sequences degenerates, as we have
        \[
                H^{s}_{\cts}(\Cong_{0},\cC(\bar K_{\gp}))
                =
                \begin{cases}
                        \cC(\tilde K_{\gp}) & \hbox{if }s=0,\\
                        0 & s>0.
                \end{cases}
        \]
        From this, it follows that
        \[
                H^{r}_{\cts}(\Gamma\times \Cong_{0},\cC(\bar K_{f}))
                =
                H^{r}(\Gamma,\cC(\tilde K_{f}))
        \]
        From the second spectral sequence, we have an inflation-restriction sequence:
        \[
                0 \to H^{1}_{\cts}(\Cong_{0},E) \to H^{1}(\Gamma, \cC(\tilde K_{f}))
                \to H^{1}(\Gamma,\cC(\bar K_{f}))^{\Cong_{0}}.
        \]
        By Theorems 5 and 6 of \cite{hill10}, we have
        \[
                H^{\bullet}(\Gamma,\cC(\bar K_{f}))
                =
                E \otimes_{\Z_{p}} \limp{t} \limd{\Upsilon} H^{\bullet}(\Upsilon,\Z/p^{t}),
        \]
        where $\Upsilon$ runs through the subgroups of finite index in $\Gamma$.
        In particular, we have
        $H^{1}(\Gamma,\cC(\bar K_{f}))=0$, and so
        \begin{equation}
                H^{1}_{\cts}(\Cong_{0},E) \cong H^{1}(\Gamma, \cC(\tilde K_{f})).
                \label{step1}
        \end{equation}
        Next, we consider $\cC(\tilde K_{f})$ as a $\Gamma \times \hat K^{\gp}$-module.
        As before, we have
        \[
                H^{s}_{\cts}(\hat K^{\gp},\cC(\tilde K_{f}))
                =
                \begin{cases}
                        \cC(\tilde K_{\gp}) & s=0,\\
                        0 & s>0.
                \end{cases}
        \]
        This implies
        \[
                H^{\bullet}(\Gamma \times \hat K^{\gp}, \cC(\tilde K_{f}))
                =
                H^{\bullet}(\Gamma , \cC(\tilde K_{\gp}))
                =
                \tH^{\bullet}(\hat K^{\gp},E).
        \]
        Hence there is a spectral sequence
        \[
                H^{r}_{\cts}(\hat K^{\gp},H^{s}(\Gamma, \cC(\tilde K_{f})))
                \Rightarrow
                \tH^{r+s}(\hat K^{\gp},E).
        \]
        The sequence of low degree terms gives:
        \[
                H^{1}_{\cts}(K^{\gp},E)
                \to \tH^{1}(\hat K^{\gp},E)
                \to H^{1}(\Gamma, \cC(\tilde K_{f}))^{\hat K^{\gp}}
                \to H^{2}_{\cts}(K^{\gp},E).
        \]
        The theorem will be proved by \eqref{step1} when we have shown that the middle map in the sequence above 
        is an isomorphism.
        We shall show that the first and last terms above are zero.
        As $E$ is a field of characteristic zero, it follows that for any normal subgroup $U$ of finite index in $K^{\gp}$,
         we have $H^{\bullet}_{\cts}(K^{\gp},E)=H^{\bullet}_{\cts}(U,E)^{K_{\gp}}$.
        It is therefore sufficient to prove that $H^{1}_{\cts}(U,E)$ and $H^{2}_{\cts}(U,E)$ are trivial for a suitable
        subgroup $U$ of finite index in $K^{\gp}$.
        Let $S$ be finite set of finite places of $k$, distinct from $\gp$, such that for every prime $\gq$ outside
        $S\cup\{\gp\}\cup\infty$, the subgroup $K_{\gq}=K^{\gp}\cap G(k_{\gq})$ is a
        hyperspecial maximal compact subgroup, and is perfect.
        For primes $\gq$ in $S$, we choose a subgroup $K_{\gq}$ contained in $K^{\gp}\cap \bG(k_{\gq})$.
        We shall take $U$ to be the subgroup $\prod_{\gq\ne \gp} K_{\gq}$.
        This can be written as $U=U^{S}\oplus U_{S}$, where $U^{S}=\prod_{\gp\ne \gq\notin S} K_{\gq}$
        and $U_{S}=\prod_{\gq\in S}K_{\gq}$.
        Since $\bG$ is semisimple, it follows that for every prime $\gq$, the commutator subgroup $[K_{\gq},K_{\gq}]$
        has finite index in $K_{\gq}$, and hence $\Hom(K_{\gq},E)=0$.
        We therefore have $H^{1}_{\cts}(U,E)=0$.
        
        To calculate $H^{2}_{\cts}(U^{S},E)$, recall the short exact sequence
of Theorem \ref{thm:standardfacts}:
        \[
                0 \to \sideset{}{^{(1)}}\varprojlim_t H^{1}(U^{S},
\Z/p^{t})
                \to H^{2}_{\cts}(U^{S},\Z_{p})
                \to \limp{t} H^{2}(U^{S}, \Z/p^{t})
                \to 0.
        \]
        By the same argument as above, the groups $H^{1}(U^{S}, \Z/p^{t})$ are zero,
         and so the first term in this short exact sequence is zero.
        We have from \cite[\S2.2, Cor. 1]{serre73b}
        \[
                H^{\bullet}(U^{S}, \Z/p^{t})
                =
                \limd{T \hbox{ finite}}
                H^{\bullet}(K_{T},\Z/p^{t}),
        \]
        where $T$ runs over the finite sets of places of $k$ which do not intersect $S\cup\{\gp\}\cup\infty$
        and $K_{T}= \prod_{\gq\in T}K_{\gq}$.
        In particular, since $K_{T}$ is a perfect group it has a universal topological central extension, whose kernel
        we shall denote $\pi_{1}(K_{T})$.
        It follows that
        \[
                H^{2}(U^{S}, \Z/p^{t})
                =
                \limd{T \hbox{ finite}}
                \Hom(\pi_{1}(K_{T}),\Z/p^{t})
                =
                \Hom\left(\Pi,\Z/p^{t}\right),
        \]
        where
        \[
                \Pi = \prod_{\gq \notin S\cup\{\gp\}\cup\infty} \pi_{1}(K_{\gq}).
        \]
        The group $\Pi$ is a product of finite groups, and so we have
        \[
                \limp{t} \Hom\left(\Pi,\Z/p^{t}\right) = 0.
        \]
        This shows that $H^{2}_{\cts}(U^{S},\Z_{p})=0$, and in particular
        $H^{2}_{\cts}(U^{S},E)=0$.

        Next let $\gq$ be a prime in $S$ which does not lie above $p$.
        In this case there is an $E$-valued Haar measure on $K_{\gq}$,
        so it follows that $H^{r}(K_{\gq},E)=0$ for all $r>0$ and in particular when $r=2$.
        
        Finally, suppose $\gq$ is a prime in $S$ which lies above $p$.
        In this case there is an isomorphism \cite[Theorem 2.4.10 of Chapter V]{lazard65}
        \[
                H^{\bullet}_{\cts}(K_{\gq},E)
                =
                \left( H^{\bullet}_{\Lie}(\gg,\Qp) \otimes E\right)^{K_{\gq}},
        \]
        where the Lie algebra $\gg$ of $K_{\gq}$ is regarded as a Lie algebra over $\Qp$.
        By Whitehead's Second Lemma \cite{weibel}
         we have $H^{2}_{\Lie}(\gg,\Qp)=0$, and so in this case
        we also have $H^{2}_{\cts}(K_{\gq},E)=0$.
        
        As a consequence, we deduce that $H^{2}_{\cts}(U,E)=0$, which finishes the proof of the theorem.
\end{proof}

In particular, if the congruence kernel is finite then $\tH^{1}=0$.
As a result, we have:

\begin{theorem}
	For any metaplectic group the edge map criterion holds in dimension $0$.
 	If $\bG$ is semi-simple, simply connected and has positive real rank,
	and $\varepsilon$ is non-trivial then the edge map criterion holds in dimension $1$.
        If in addition $\bG$ has finite congruence kernel, then the edge map criterion holds in dimension 2.
\end{theorem}

\begin{proof}
	In dimension $0$ the edge map is clearly an isomorphism, since it is
	in the bottom left corner of the spectral sequence.
	If $\bG$ is semi-simple, simply connected and has positive real rank,
	and $\varepsilon$ is non-trivial then we've seen that $\tH^{0}_{\varepsilon}=0$,
	and so the edge map is an isomorphism in dimension $1$.
	When the congruence kernel is finite we also have $\tH^{1}_{\varepsilon}=0$,
	and so the edge map is an isomorphism in dimension $2$.
	
\end{proof}



\section{The $p$-adic metaplectic Jacquet functor}
\label{sect:jacquet}

The results of this section are of a local nature, so we shall alter our
notation.
We now suppose that
we have a connected reductive group $\cG$ defined over $\Qp$, and we write $\scG$ for
the group of $\Qp$-valued points.
We shall suppose also that we have a topological central extension
\[
        1 \to \mu\to \tilde \scG \stackrel{\pr}{\to} \scG \to 1,
\]
where $\mu$ is a finite abelian group.
There is a compact open subgroup $\scK$ of $\scG$ which lifts to a subgroup
$\hat\scK$ of $\tilde\scG$.

Let $\cP$ be a parabolic subgroup of $\cG$ defined over $\Qp$ with unipotent radical $\cN$, and
choose a Levi component $\cM$. We shall also write $\scP$, $\scM$ and $\scN$ for the
groups $\cP(\Qp)$, $\cM(\Qp)$ and $\cN(\Qp)$ respectively. We write $\tilde \scP$,
$\tilde \scM$ and $\tilde \scN$ for the preimages of $\scP$, $\scM$ and $\scN$ in $\tilde \scG$.

\begin{lemma}
        Let $\scN$, $\scM$ and $\scP$ be as above.
        \begin{itemize}
                \item[(a)]
                There is a unique subgroup $\hat\scN$ of $\tilde \scN$,
                such that $\hat\scN$ projects bijectively onto $\scN$.
                \item[(b)]
                The subgroup $\hat\scN$ is open (and hence closed) in $\tilde\scN$
                and normal in $\tilde\scP$.
                Furthermore we have $\tilde\scP=\tilde\scM\ltimes\hat\scN$.
                \item[(c)]
                \label{conjugate-lift}
                Let $\tau :\scN \to \hat \scN$ be the unique splitting of $\pr:\tilde \scN \to \scN$.
                For any $\tilde m \in \tilde \scM$ and any $n\in \scN$, we have
                \[
                        \tilde m^{-1} \tau(n) \tilde m
                        =
                        \tau(m^{-1} n m),
                \]
                where $m=\pr(\tilde m)$.
        \end{itemize}
%
\end{lemma}

\begin{proof}
        For the moment we shall regard $\Qp$ as a discrete additive group.
        As such, $\Qp$ is uniquely divisible, and so we have $H^{\bullet}_{\group}(\Qp,\mu)=\mu$.
        By this we mean that $H^{0}_{\group}(\Qp,\mu)=\mu$ and $H^{n}_{\group}(\Qp,\mu)=0$ for $n>0$.
        Suppose we have a central extension of discrete groups
        \[
                1 \to \Qp \to \scN_{1} \to \scN_{2} \to 1.
        \]
        It follows by the Hochschild--Serre spectral sequence that
        $H^{\bullet}_{\group}(\scN_{1},\mu)=H^{\bullet}_{\group}(\scN_{2},\mu)$.
        The group $\scN$ may be constructed from a sequence of central extensions by $\Qp$
        of the trivial group.
        Therefore $H^{\bullet}_{\group}(\scN,\mu)=\mu$, and in particular $H^{2}_{\group}(\scN,\mu)=0$.
        This shows that the extension splits on $\scN$, and hence shows the existence of $\hat\scN$.
        
        For uniqueness, suppose that $\tau_{1},\tau_{2}:\scN \to \tilde \scN$ are two splittings.
        It follows that $n \mapsto \tau_{1}(n) \tau_{2}(n)^{-1}$ is a homomorphism from $\scN$ to $\mu$.
        Since $\scN$ is divisible, this homomorphism must be trivial, so $\tau_{1}=\tau_{2}$.
        
        To see that $\hat\scN$ is open and closed in $\tilde \scN$, we note that the same calculation as above
        proves for the continuous cohomology $\scN$ that $H^{\bullet}_{\cts}(\scN,\mu)=\mu$.
        In particular, the section $\tau:\scN \to \hat \scN$ is continuous (and hence a homeomorphism).
        
        As $\scN$ is normal in $\scP$, it follows that $\tilde \scN$ is normal in $\tilde \scP$.
        The uniqueness property of $\hat\scN$ shows that $\hat\scN$ is normal in $\tilde \scP$.
        
        Part (c) also follows from the uniqueness of $\hat\scN$.
\end{proof}

%
%
%
%
%

Now let $\scN_{0}$ be a compact open subgroup of $\scN$, and
let $\hat \scN_{0}=\tau (\scN_{0})$.
Following \S3.3 of \cite{emerton-jacquet1} we define two semigroups:
\[
        \scM^{+}
        =
        \{m \in \scM : m\scN_{0}m^{-1} \subset \scN_{0}\},
\]
and
\[
        \tilde \scM^{+}
        =
        \{\tilde m \in \tilde \scM : \tilde m\hat \scN_{0}\tilde m^{-1} \subset \hat \scN_{0}\}.
\]

\begin{lemma}
        \label{lem-M}
        With the notation described above, $\tilde \scM^{+}$ is the
        preimage of $\scM^{+}$ in $\tilde \scM$.
\end{lemma}

\begin{proof}
        This is immediate from the fact that $\hat \scN$ is a normal subgroup of $\tilde \scP$.
\end{proof}

Let $Z_{\scM}$ and $Z_{\tilde \scM}$ be the centres of $\scM$ and $\tilde \scM$ respectively.

\begin{lemma}
        \label{finite-index-centre}
        The image of $Z_{\tilde \scM}$ in $\scG$ is a subgroup of $Z_{\scM}$ of finite index.
\end{lemma}

\begin{proof}
        Let $\tilde z \in \pr^{-1}(Z_{\scM})$ and $\tilde m\in \tilde \scM$. Furthermore let $z=\pr(\tilde z)$
         and $m=\pr(\tilde m)$.
        Since $Z_{\scM}$ is central in $\scM$, we have $[z,m]=1$, and therefore $[\tilde z,\tilde m]\in \mu$.
        Since our extension is central, it follows that the commutator $[\tilde z,\tilde m]$ depends only
        on $z$ and $m$. Furthermore, one easily checks that the map $Z_{\scM}\times \scM \to \mu$
        given by
        \[
                (z,m) \mapsto [\tilde z,\tilde m]
        \]
        is bimultiplicative.
        In particular, if $z$ is an $|\mu|$-th power, then $[\tilde z,\tilde m]=1$ for all $\tilde m$.
        This shows that the projection of $Z_{\tilde \scM}$ contains $Z_{\scM}^{|\mu|}$.
        Since $Z_{\scM}$ is topologically finitely generated, it follows that  $Z_{\scM}^{|\mu|}$
        has finite index in $Z_{\scM}$.
        This proves the lemma.
\end{proof}

\begin{remark}
        The projection of $Z_{\tilde \scM}$ is typically not equal to $Z_{\scM}$.
        For example, suppose $\scG=\GL_{2}(\Qp)$.
        Assume $\Qp$ contains an $m$-th root of unity.
        Then Kubota has defined an $m$-fold cover $\tilde \scG$ of $\scG$.
        We obviously have $Z_{\scG}=\Qp^{\times}$.
        On the other hand, the image in $Z_{\tilde\scG}$ of $Z_{\scG}$ is
        \[
                \{x\in \Qp^{\times}: \forall y\in \Qp^{\times}, (x,y)_{\gp,m}=(y,x)_{\gp,m}\}.
        \]
        This is the set of $x$ such that $x^{2}$ is an $m$-th power in $\Qp^{\times}$.
\end{remark}

\begin{lemma}
        The group $\tilde \scM$ is generated as a semigroup by $\tilde \scM^{+}$ and $Z_{\tilde \scM}$.
\end{lemma}

\begin{proof}
        Choose any $\tilde m\in \tilde \scM$ and let $m=\pr(\tilde m)$.
        Since $\pr(Z_{\tilde \scM})$ has finite index in $Z_{\scM}$, Lemma 3.3.1 of \cite{emerton-jacquet1} shows
        that there is an element $z \in \pr(Z_{\tilde \scM})$,
        such that $zm\scN_{0} m^{-1}z^{-1}\subset \scN_{0}$.
        Hence $zm\in \scM^{+}$.
        If $\tilde z$ is any preimage of $z$ in $Z_{\tilde \scM}$, then by Lemma \ref{lem-M} we have
        $\tilde z\tilde m\in \tilde \scM^{+}$.
\end{proof}

\subsection{Definition of the Jacquet functor}

We shall consider representations of the group $\tilde \scG$ over a coefficient
field $E$ containing $\Qp$.
Let $(V,\pi)$ be a locally analytic representation of $\tilde \scP$.
In this section we shall define the Jacquet functor $J_{\cP}(V)$.

Fix a compact open subgroup $\scP_{0}$ of $\scP$ and let $\tilde \scP_{0}$ be the preimage
of $\scP_{0}$ in $\tilde \scG$.
Let $\hat \scN_{0}$ and $\tilde \scM_{0}$ be the intersections of $\tilde \scP_{0}$ with $\tilde \scM$ and $\hat \scN$
respectively.
We also define two semigroups:
\[
        \tilde \scM^{+}
        =
        \{m \in \tilde \scM : m \hat \scN_{0}m^{-1} \subset \hat \scN_{0}\},
        \quad
        \tilde Z^{+}= \tilde \scM^{+}\cap Z_{\tilde \scM}.
\]
Recall that the subspace $V^{\hat \scN_{0}}$ has a natural action
of $\tilde \scM^{+}$. This action is defined as follows. For $m\in \tilde \scM^{+}$ and $v\in V^{\hat \scN_{0}}$,
the vector $\pi(m)v$ will be in $V^{m \hat \scN_{0}m^{-1}}$, and we define
\[
        \left(\pi_{\hat \scN_{0}}(m)\right)(v)
        =
        \int_{\hat \scN_{0}} \pi(n m) v \ dn,
\]
where the Haar measure on $\hat \scN_{0}$ is normalized to have total measure $1$.
For elements $m\in \tilde \scM_{0}$ we have $\pi_{\hat \scN_{0}}(m)v=\pi(m)v$,
and so the action $\pi_{\hat \scN_{0}}$ of $\tilde \scM^{+}$ on $V^{\hat \scN_{0}}$ is locally analytic.

Let $\widehat{Z}_{\tilde \scM}$ be the rigid analytic space of locally analytic characters of $Z_{\tilde \scM}$,
and write $\cC^{\an}(\widehat{Z}_{\tilde \scM},E)$ for the ring of $E$-valued analytic functions on
$\widehat{Z}_{\tilde \scM}$.

\begin{definition}
        If $V$ is a locally analytic representation of $\tilde \scM^{+}$ then we define
        the \emph{finite slope} part of $V$ by
        \[
                V_{\fs}
                =
                \cL_{b,\tilde Z^{+}}(\cC^{\an}(\widehat{Z}_{\tilde \scM},E),V).
        \]
\end{definition}

There is a natural map $Z_{\tilde \scM} \to \cC^{\an}(\widehat{Z}_{\tilde \scM},E)$, which makes
$V_{\fs}$ into a $Z_{\tilde \scM}$-module.
Furthermore, the action of $\tilde \scM^{+}$ on $V$ gives rise to an action of $\tilde \scM^{+}$ on $V_{\fs}$.
The actions of $Z_{\tilde \scM}$ and $\tilde \scM^{+}$ coincide on their intersection $Z_{\tilde \scM}^{+}$,
and so generate an action of the group $\tilde \scM = \tilde \scM^{+} Z_{\tilde \scM}$ on $V_{\fs}$.

\begin{definition}
        If $V$ is a locally analytic representation of $\tilde \scP$, then we define the Jacquet functor of $V$
        by
        \[
                J_{\cP}(V)
                =
                \left(V^{\hat \scN_{0}}\right)_{\fs}.
        \]
\end{definition}

\subsection{The Jacquet functor preserves essential admissibility}

\begin{theorem}
        \label{thm:essential-admiss}
        If $V$ is an essentially admissible locally analytic representation of $\tilde \scG$
        then $J_{\cP}(V)$ is an essentially admissible locally analytic representation
        of $\tilde \scM$.
\end{theorem}

Theorem \ref{thm:essential-admiss} was proved by Emerton in the algebraic case. The proof
is rather long, and most of it carries through word for word to the metaplectic case.
The only difference is a technical lemma on the structure of the group $\scG$.
We quote this lemma below, and we prove its generalization to the metaplectic case.

Before stating these lemmata, we must recall the definition of a rigid analytic Iwahori decomposition,
and generalize this concept to the metaplectic case.
Let $\cP$ and $\bar\cP$ be parabolic subgroups of $\cG$ defined over $\Qp$,
so that $\cM=\cP \cap \bar\cP$ is a Levi component of both $\cP$ and $\bar\cP$.
We shall write $\cN$ and $\bar\cN$ for the unipotent radicals of $\cP$ and $\bar\cP$ respectively.
We shall write $\gn$, $\bar\gn$ and $\gm$ for the Lie algebras of $\cN$, $\bar\cN$ and $\cM$
over $\Qp$, and $\scN$, $\bar \scN$ and $\scM$ for their groups of $\Qp$-valued points.
Suppose $\scH$ is a good analytic open subgroup of $\scG$, which arises from the Lie
sublattice $\gh$ of $\gg$, with underlying rigid analytic group $\bbH$.
Furthermore define
\[
        \scM_{0}= \scM \cap \scH,
        \quad
        \scN_{0}= \scN \cap \scH,
        \quad
        \bar \scN_{0}= \bar \scN \cap \scH.
\]
Finally, we let $\M_{0}$, $\N_{0}$ and $\bar\N_{0}$ denote the rigid analytic closures of
$\scM_{0}$, $\scN_{0}$ and $\bar \scN_{0}$ in $\bbH$.
The subgroup $\scH$ is said to admit a \emph{rigid analytic Iwahori decomposition}
if the following conditions are satisfied:
\begin{enumerate}
        \item
        The groups $\scM_{0}$, $\scN_{0}$ and $\bar \scN_{0}$ are good analytic open subgroups of
        $\scM$, $\scN$ and $\bar \scN$ corresponding to the Lie sublattices $\gm\cap\gh$, $\gn \cap \gh$
        and $\bar\gn \cap \gh$, and with underlying rigid analytic groups $\M_{0}$, $\N_{0}$ and
        $\bar\N_{0}$.
        \item
        The rigid analytic map
        \[
                \bar \N_{0}\times \M_{0}\times \N_{0} \to \bbH
        \]
        given by multiplication in $\bbH$ is an isomorphism of rigid analytic spaces.
\end{enumerate}
We next give a corresponding definition for subgroups of $\tilde\scG$.
Note that we have a compact open subgroup $\scK$ of $\scG$, which lifts to a subgroup
$\hat \scK$ of $\tilde \scG$, and we also have a unique lifts $\hat \scN$ and $\hat{\bar\scN}$ of $\scN$
and $\bar \scN$ to $\tilde \scG$.
These lifts do not necessarily coincide on $\scK\cap \scN$ and $\scK\cap \bar \scN$.
However, by reducing the size of $\scK$ is necessary, we may assume that
\[
        \hat \scK \cap \hat \scN
        \stackrel{\pr}{ \cong}
        \scK\cap \scN,
        \quad
        \hat \scK \cap \hat{\bar \scN}
        \stackrel{\pr}{ \cong}
        \scK\cap \bar \scN.
\]

\begin{definition}
        Suppose $\hat \scK$ and $\scK$ are chosen to satisfy these conditions above.
        We say that a good analytic subgroup $\scH$ of $\tilde\scG$ has a rigid analytic Iwahori decomposition
        with respect to $\cP$ and $\bar\cP$ if (i) $\scH$ is contained in $\hat\scK$
        and (ii) the image of $\scH$ in $\scG$ has a rigid analytic Iwahori decomposition with respect to $\cP$ and $\bar\cP$.
\end{definition}

Recall the following technical result of Emerton:

\begin{proposition}\cite[Prop. 4.1.6]{emerton-jacquet1}
        \label{emerton-4.1.6}
        We may find a decreasing sequence 
        $\{\scH_{n}\}_{n\ge 0}$ of good analytic 
        open subgroups of $\scG$, cofinal in the directed set of all
         analytic open subgroups of $\scG$, 
        and satisfying the following conditions: 
        \begin{itemize}
                \item[(i)]
                For each $n\ge 0$, the inclusion $\scH_{n+1}\subset \scH_{n}$
                 extends to a relatively compact rigid analytic map
                 $\bbH_{n+1} \subset \bbH_{n}$ .
                 \item[(ii)]
                 For each $n\ge 0$, the subgroup $\scH_{n}$ of $\scH_{0}$ is normal.
        \end{itemize}
        Let $\cP_{\emptyset}$ be a minimal parabolic subgroup of $\cG$ defined over $\Qp$.
        The remaining properties refer to any pair $\cP$ and $\bar \cP$ of opposite parabolic subgroups of $\cG$,
         chosen so that $\cP$ contains $\cP_{\emptyset}$ and $\bar\cP$ contains $\bar\cP_{\emptyset}$.
        \begin{itemize}
                \item[(iii)]
                Each $\scH_{n}$ admits a rigid analytic Iwahori decomposition with respect to $\cP$ and $\bar\cP$.
                \item[(iv)]
                If $z\in Z_{\scM}$ is such that $z^{-1} \bar \scN_{0} z \subset \bar \scN_{0}$,
                 then $z^{-1}\bar \scN_{n} z \subset \bar \scN_{n}$ for each $n \ge 0$.
                \item[(v)]
                If $z\in Z_{\scM}$ is such that $z \scN_{0} z^{-1}  \subset \scN_{0}$,
                 then $z \scN_{n} z^{-1} \subset \scN_{n}$ for each $n \ge 0$.
                \item[(vi)]
                We may find $z \in Z_{\scM}$ such that
                $z^{-1}\bar \scN_{0} z \subset \bar \scN_{0}$ and $z \scN_{0} z^{-1}  \subset \scN_{0}$,
                and such that, for each $n\ge 0$,
                the embedding of part (iv) factors through the inclusion
                $\bar \scN_{n+1}\subset \bar \scN_{n}$. 
        \end{itemize}
\end{proposition}

In order to prove Theorem \ref{thm:essential-admiss}, it is sufficient to prove the following result analogous to Proposition \ref{emerton-4.1.6}.
The rest of the proof of the theorem is word for word the same as in \cite{emerton-jacquet1}.

\begin{proposition}
        \label{emerton-4.1.6-met}
        We may find a decreasing sequence 
        $\{\scH_{n}\}_{n\ge 0}$ of good analytic 
        open subgroups of $\tilde \scG$, cofinal in the directed set of all
         analytic open subgroups of $\tilde \scG$, 
        and satisfying the following conditions: 
        \begin{itemize}
                \item[(i)]
                For each $n\ge 0$, the inclusion $\scH_{n+1}\subset \scH_{n}$
                 extends to a relatively compact rigid analytic map
                 $\bbH_{n+1} \subset \bbH_{n}$ .
                 \item[(ii)]
                 For each $n\ge 0$, the subgroup $\scH_{n}$ of $\scH_{0}$ is normal.
        \end{itemize}
        Let $\cP_{\emptyset}$ be a minimal parabolic subgroup of $\cG$ defined over $\Qp$.
        The remaining properties refer to any pair $\cP$ and $\bar \cP$ of opposite parabolic subgroups of $\cG$,
         chosen so that $\cP$ contains $\cP_{\emptyset}$ and $\bar\cP$ contains $\bar\cP_{\emptyset}$.
        \begin{itemize}
                \item[(iii)]
                Each $\scH_{n}$ admits a rigid analytic Iwahori decomposition with respect to $\cP$ and $\bar\cP$.
                \item[(iv)]
                If $\tilde z\in Z_{\tilde \scM}$ is such that $\tilde z^{-1} \hat {\bar \scN}_{0} \tilde z \subset \hat{\bar \scN}_{0}$,
                 then $\tilde z^{-1} \hat{\bar \scN}_{n} \tilde z \subset \hat{\bar \scN}_{n}$ for each $n \ge 0$.
                \item[(v)]
                If $\tilde z\in Z_{\tilde \scM}$ is such that $\tilde z\hat \scN_{0} \tilde z^{-1}  \subset \hat \scN_{0}$,
                 then $\tilde z \hat \scN_{n} \tilde z^{-1} \subset \scN_{n}$ for each $n \ge 0$.
                \item[(vi)]
                We may find $\tilde z \in Z_{\tilde \scM}$ such that
                $\tilde z^{-1} \hat{\bar \scN}_{0} \tilde z \subset \hat{\bar \scN}_{0}$
                and $\tilde z\hat \scN_{0} \tilde z^{-1}  \subset \hat \scN_{0}$,
                and such that, for each $n\ge 0$,
                the embedding of part (iv) factors through the inclusion
                $\hat{\bar \scN}_{n+1}\subset  \hat{\bar \scN}_{n}$. 
        \end{itemize}
\end{proposition}

\begin{proof}
        Recall that we have a compact open subgroup $\scK$ of $\scG$,
        which lifts to a subgroup $\hat\scK$ of $\tilde \scG$.
        Furthermore, $\scK$ and $\hat\scK$ are chosen small enough so that
        for each standard parabolic subgroup $\cP=\cM\cN$,
        the lift $\scK \to \hat\scK$    coincides with the lift $\scN \to \hat \scN$
        (resp. $\bar \scN \to \hat{\bar \scN}$)
        on $\scK\cap \scN$ (resp. $\scK\cap \bar\scN$).
        Emerton's proof of Proposition \ref{emerton-4.1.6} actually shows a little bit more than is stated.
        He in fact shows that we may in addition take $\scH_{0}$ to be arbitrarily small.
        We may therefore take a sequence of subgroups $\scH_{n}$ satisfying Proposition \ref{emerton-4.1.6}
        with $\scH_{0}$ contained in $\scK$.
        We then define a new sequence of subgroups $\hat\scH_{n}$ in $\tilde\scG$, where each
        $\hat \scH_{n}$ is the lift of $\scH_{n}$ to $\hat\scK$.
        We claim that the sequence $\hat \scH_{n}$ satisfies Proposition \ref{emerton-4.1.6-met}.
        Properties (i), (ii) and (iii) for $\hat\scH_{n}$ are clear,
        since they only depend on the original groups $\scH_{n}$.
        We next consider property (iv). For an element $\tilde z\in Z_{\tilde \scM}$, we shall write
        $z$ for the image of $z$ in $Z_{\scM}$.
        Lemma \ref{conjugate-lift} (c) shows that the equation
        $\tilde z^{-1} \hat{\scN}_{n} \tilde z \subset \hat{\scN}_{n}$ is equivalent to
        $z^{-1} \scN_{n} z \subset \scN_{n}$.
        Hence property (iv) of Proposition \ref{emerton-4.1.6-met} is a consequence of property (iv)
        of Proposition \ref{emerton-4.1.6}.
        Similarly, property (v) of Proposition \ref{emerton-4.1.6-met} follows from the corresponding
        property in Proposition \ref{emerton-4.1.6}.
        Suppose that $z \in Z_{\scM}$ is chosen to satisfy property (vi) of Proposition  \ref{emerton-4.1.6}.
        It follows that every power $z^{r}$ ($r>0$) also has this property.
        Furthermore by Lemma \ref{finite-index-centre}, there is a suitable power $z^{r}$ which is
        also in $\pr(Z_{\tilde\scM})$.
        We replace $z$ by such a power and let $\tilde z$ be a pre-image in $Z_{\tilde\scM}$ of $z$.
        Again using Lemma \ref{conjugate-lift} (c), we deduce that $\tilde z$ has property (vi)
        of Proposition \ref{emerton-4.1.6-met}.
\end{proof}

Exactly as in \cite{hillloeffler11}, we may strengthen Theorem \ref{thm:essential-admiss} as follows. Let $\cD$ be the derived subgroup of $\cM$ (a semisimple algebraic group over $\Qp$), and $\scD = \cD(\Qp)$.
Then if $\mathscr{D}_0$ is an open compact subgroup of $\scD$ which lifts to
a subgroup $\hat\scD_{0}$ of $\tilde\scD$,
and $W$ a finite-dimensional continuous representation of $\hat\scD_0$,
we may consider the representation
\[
        \left(J_\cP(V) \otimes W\right)^{\hat{\mathscr{D}}_0}
\]
of $Z_{\tilde \scM}$. This representation is essentially admissible, by proposition 3.3 of \cite{hillloeffler11}. Since $Z_{\tilde \scM}$ is commutative, this space corresponds to a coherent sheaf on the character space $\widehat{Z_{\tilde \scM}}$. Let $\Sigma$ be the support of this sheaf. Differentiation of characters gives a map $\widehat{Z_{\tilde \scM}} \to \check{\mathfrak{z}}$, where $\mathfrak{z}$ is the Lie algebra of $Z_\cM$ over $E$ and $\check{\mathfrak{z}}$ its dual space.

\begin{theorem}\label{thm:discretefibres}
If $V$ is admissible, then the map $\Sigma \to \check{\mathfrak{z}}$ has discrete fibres.
\end{theorem}

If $\cP = \cB$ is a Borel subgroup, so $\cD$ is trivial, then this is a metaplectic analogue of \cite[Proposition 4.2.23]{emerton-jacquet1}, and the proof also carries over identically using the family of subgroups $\scH_n$ constructed above. In the general case, one need only add the requirement that the subgroups $\scM_n$ have rigid-analytic decompositions as products of subgroups of $\mathscr{D}$ and $Z_{\scM}$; then the proof proceeds exactly as in \cite[\S 4.3]{hillloeffler11}.

\subsection{The Jacquet functor of an admissible smooth representation}

In this section, we consider a smooth admissible representation $V$ of $\tilde\scG$. The classical theory of the Jacquet functor for smooth representations of algebraic groups applies equally to metaplectic covers such as $\tilde\scG$ (see \cite[\S 6]{mcnamara10}). Recall that for such a representation, the classical Jacquet functor is defined to be
the module of $\hat\scN$-coinvariants $V_{\hat\scN}$. This is the largest quotient of $V$ on
which $\hat\scN$ acts trivially, and is a smooth representation of $\tilde \scM$.
We may however regard $V$ as a locally analytic representation, so we also have the locally
analytic Jacquet functor $J_{\cP}(V)$ defined above.
In this section we show that $J_{\cP}(V)$ is canonically isomorphic to $V_{\hat\scN}$.

For a smooth representation $V$ of $\tilde \scG$, we shall always regard the vector space $V$
as a topological vector space with the finest locally convex topology.
In this topology, \emph{every} vector subspace $V' \subseteq V$ is closed, and the
subspace topology on $V'$ is again the finest locally convex topology.
This has the following consequence: if $U$ is a Fr\'echet space, and $f : U
\to V$ is a continuous linear map, then $f$ must have finite rank. This is
because the space $\operatorname{Coim}(f) = U / \ker(f)$, with its quotient
topology, is a Fr\'echet space; but it maps continuously and bijectively to
$\im(f)$, which has the finest locally convex topology. Hence this
map is a topological isomorphism, and we see that the finest locally convex
topology on $\operatorname{Im}(f)$ is Fr\'echet, which can only happen if
$\operatorname{Im}(f)$ is finite-dimensional.

\begin{theorem}
        \label{smooth-jacquet}
        Let $V$ be an admissible smooth representation of $\tilde \scG$.
        Then there is a canonical isomorphism
        \[
                J_{\cP}(V)
                \cong
                V_{\hat \scN},
        \]
        where $V_{\hat \scN}$ is the space of $\hat \scN$-coinvariants of $V$.
\end{theorem}

\begin{proof}
        The proof of this theorem is exactly the same as in the algebraic case,
        which is dealt with in \cite{emerton-jacquet1}. We shall merely recall the main steps.
        Composing the inclusion $V^{\hat \scN_{0}} \to V$ with the projection $V\to V_{\hat \scN}$
        we get a canonical map $\Phi:V^{\hat \scN_{0}}\to V_{\hat \scN}$, and this map is $\tilde \scM^{+}$-equivariant.
        Using the smoothness of $V$, we can show that $\Phi$ is also surjective:
        indeed for any $v\in V$ we may define
        \[
                v'
                =
                \pi_{\hat \scN_{0}}(v)
                =
                \frac{1}{|\hat \scN_{0}|} \int_{\hat \scN_{0}} \pi(n) v\ dn.
        \]
        Clearly $v'$ is in $V^{\hat \scN_{0}}$, and has the same image in $V_{\hat \scN}$ as $v$.
        The kernel of $\Phi$ consists of those vectors $v\in V^{\hat \scN_{0}}$ for which there is a sufficiently large compact open subgroup
        $\hat \scN_{1}\subset \hat \scN$, for which $\pi_{\hat \scN_{1}}(v)=0$.
        
        We call a vector $v\in V^{\hat \scN_{0}}$ \emph{null} if there is a $z\in Z_{\tilde \scM}^{+}$,
        such that $\pi_{\hat \scN_{0}}(z)(v)=0$.
        One easily checks that the kernel of $\Phi$ consists of the null vectors in $V^{\hat \scN_{0}}$.
        To complete the proof, it suffices to show that
        $V^{\hat \scN_{0}}= (V^{\hat \scN_{0}})_{\Null} \oplus (V^{\hat \scN_{0}})_{\fs}$.
        
        Using the fact that $V$ is an admissible smooth representation, one can show that
        each $v\in V^{\hat \scN_{0}}$ is contained in a finite dimensional $Z_{\tilde \scM}^{+}$-invariant subspace $W$.
        We therefore have $V^{\hat \scN_{0}}=\limd{} W$ with $W$ ranging over such subspaces.
        Since $\cC^{\an}(\widehat Z_{\tilde \scM})$ is a Fr\'echet space, the remark preceding the theorem shows that
        $(V^{\hat \scN_{0}})_{\fs}=\limd{} W_{\fs}$.
        Furthermore it is clear that $(V^{\hat \scN_{0}})_{\Null}=\limd{} W_{\Null}$.
        It follows by elementary linear algebra, that $W=W_{\fs}\oplus W_{\Null}$ for any finite dimensional
        representation $W$ of $Z_{\tilde \scM}^{+}$, and so the result follows.
\end{proof}

More generally, we have the following result applying to locally algebraic
representations of $\tilde \scG$:

\begin{theorem}
        Let $V$ be an admissible smooth representation of $\tilde \scG$
        and let $W$ be a finite dimensional irreducible algebraic representation
        of $\cG$, which we shall regard as a representation of $\tilde \scG$, trivial
        on $\mu$.
        Then there is a canonical isomorphism
        \[
                J_{\cP}(V\otimes W)
                \cong
                V_{\hat \scN} \otimes W^{\scN}.
        \]
\end{theorem}

\begin{proof}
        The proof of this result in the algebraic case (Proposition 4.3.6 of \cite{emerton-jacquet1})
        works in the metaplectic case. We shall recall some details here.
        
        Let $\gn$ be the Lie algebra of $\hat \scN$.
        The action of $\gn$ on $V$ is trivial, and so we have
        $(V\otimes W)^{\gn}=V \otimes W^{\gn}$.
        On the other hand, since the action of $\hat \scN$ on $W$ is algebraic (and $\hat \scN$ is connected)
        we have $W^{\gn}=W^{\hat \scN_{0}}$.
        In particular, the action of $\hat \scN_{0}$ on $W^{\gn}$ is trivial, so we have
        $(V \otimes W^{\gn})^{\hat \scN_{0}}= V^{\hat \scN_{0}} \otimes W^{\hat \scN}$.
        This implies $(V\otimes W)^{\hat \scN_{0}}=V^{\hat \scN_{0}}\otimes W^{\scN}$.
        The action of $\tilde \scM^{+}$ on $W^{\hat \scN}$ is the restriction of the usual action of $\scM$.
        Hence by proposition 3.2.9 of \cite{emerton-jacquet1} (which is stated in sufficient generality for our needs),
        it follows that $J_{\cP}(V\otimes W)= (V^{\hat \scN_{0}})_{\fs}\otimes W^{\scN}$.
        The result now follows from Theorem \ref{smooth-jacquet}.
\end{proof}

\begin{corollary}
        \label{loc-alg-jacquet}
        Suppose that $\cG$ is quasi-split over $\Qp$ and let $\cB=\cM\cN$ be a Borel
        subgroup defined over $\Qp$.
        Assume also that $\cG$ splits over $E$.
        Let $V$ be a admissible smooth representation of $\tilde \scG$
        and let $W_{\psi}$ be the finite dimensional irreducible algebraic representation
        of $\cG$ with highest weight $\psi$ with respect to $\cB$.
        Then there is a canonical isomorphism of representations of $\tilde\scM$:
        \[
                J_{\cB}(V\otimes W_{\psi})
                \cong
                V_{\hat \scN} \otimes \psi.
        \]
\end{corollary}

\begin{proof}
        This is just a special case of the previous result.
\end{proof}

\subsection{Small slope vectors}

In this section, we'll assume for simplicity 
that $\cG$ is split over the coefficient field $E$, so every irreducible
algebraic representation of $\cG$ over $E$ is absolutely irreducible.
We shall write $\bZ_{\cM}$ for the centre of $\cM$; the group $\bZ_{\cM}$ is a
torus, and we write $\cS$ for the maximal subtorus of $\bZ_{\cM}$ which splits over $\Qp$.
Let $\ord$ denote the valuation on $\overline{\Q}_{p}$
mapping $p$ to 1. 

Let $\chi : Z_{\tilde \scM} \to E^\times$ be a continuous (hence locally $\Qp$-analytic)
character.
The homomorphism $Z_{\tilde \scM} \to \Q $ given by $t
\mapsto \ord(\chi(t))$ clearly factors through the projection to $Z_\scM$. Since
the image of $Z_{\tilde \scM}$ has finite index in $Z_\scM$ (and $\Q $ is
uniquely divisible) this extends uniquely to a linear functional on
$\Q  \otimes_{\Z} \left(Z_\scM / (Z_\scM)_0\right)$, where $(Z_\scM)_0$ is
the maximal compact subgroup of $Z_\scM$.

As in \cite[\S 1.4]{emerton-jacquet1}, we may identify $\Q 
\otimes_{\Z} \left(Z_\scM / (Z_\scM)_0\right)$ with $\Q 
\otimes_{\Z} Y_\bullet$, where $Y_\bullet$ is the cocharacter group of
$\cS$. The linear functional constructed above
thus defines an element of $\Q  \otimes_\Z Y^\bullet$, where
$Y^\bullet$ is the character group of the maximal split subtorus; and as in
\textit{op.cit.} we may define $\slope(\chi) \in \Q 
\otimes_{\Z } Y^\bullet$ to be this element.

Let us write $\Delta(\cG, \cS)$ for the set of positive restricted roots of $\cS$ in $\cG$ (that is, the set of characters of $\cS$ appearing in the adjoint action on $\Lie(\cN)$). We write $R$ for the sublattice of $Y^\bullet$ generated by $\Delta(\cG, \cS)$, which is not necessarily of full rank, and $(\Q  \otimes_\Z R)^{\ge 0}$ for the $\Q^{\ge 0}$-cone in $\Q  \otimes_\Z Y^\bullet$ generated by $\Delta(\cG, \cS)$. Finally, we let $\rho$ denote the weighted half-sum of $\Delta(\cG, \cS)$, i.e. half the sum of the characters of $\cS$ appearing in the adjoint action on $\cN$ weighted by their multiplicities.

The usefulness of these definitions arises from the following two lemmas, generalising Lemmas 4.4.1 and 4.4.2 of \cite{emerton-jacquet1} to the metaplectic case. Recall that we defined $Z_{\tilde \scM}^+ = Z_{\tilde \scM} \cap \tilde \scM^+$.

\begin{lemma}\label{lemma:+veslope}
 We have $|\chi(a)| \le 1$ for all $a \in Z_{\tilde \scM}^+$ if and only if $\slope(\chi) \in (\Q  \otimes_\Z R)^{\ge 0}$.
\end{lemma}

\begin{proof}
 It suffices to note that the projection of $Z_{\tilde \scM}$ has finite index in $Z_{\scM}$; thus the projection of $Z_{\tilde \scM}^+$ is cofinal with $Z_{\scM}^+$, and hence the proof given in \cite{emerton-jacquet1} extends to the metaplectic case also.
\end{proof}

Using this in place of Lemma 4.4.1 of \cite{emerton-jacquet1}, we deduce the following analogue of Lemma 4.4.2 of \emph{op.cit.}:

\begin{lemma}\label{lemma:+veslope2}
 If a locally analytic representation $V$ of $\tilde \scP$ admits a norm which is $\tilde \scP$-invariant,
  and $\chi \in \widehat {Z_{\tilde \scM}}$ is such that
 \[ (V^{\hat \scN_0})[Z_{\tilde \scM}^+ = \chi] \ne 0,\]
 then $\rho + \slope(\chi) \in (\Q  \otimes_\Z R)^{\ge 0}$.
\end{lemma}

We now recall what is meant by an element of $\Q  \otimes_\Z Y^\bullet$ being of \emph{non-critical slope}. We write $\Delta(\cG, Z_{\cM})$ for the set of positive restricted roots of $Z_\cM$ (that is, the set of characters $\alpha$ of $Z_\cM$ appearing in the adjoint action on $\Lie \cN$). 

By hypothesis, $\cG$ is split over $E$, so we may choose a Borel subgroup $\cB$ of $\cG$ defined over $E$. We can and do assume that the unipotent radical $\cN'$ of $\cB$ contains $\cN_E$, and we choose a Levi factor $\cT$ of $\cB$ such that $Z_{\cM} \subseteq \cT \subseteq \cM$. We define $\Delta(\cG, \cT)$ as the set of characters of $\cT$ appearing in the adjoint action on $\mathfrak{n}' = \Lie \cN'$. As shown in \cite[\S 1.4]{emerton-jacquet1}, for any simple positive restricted root $\alpha \in \Delta(\cG, Z_{\cM})$, there is a unique simple positive root $\tilde \alpha \in \Delta(\cG, \cT)$ with $\tilde \alpha |_{\bZ_{\cM}} = \alpha$. To $\tilde\alpha$ is attached an element $s_{\tilde\alpha}$ of the Weyl group $\mathop{W}(\cG, \cT)$. We define $\rho$ to be the weighted half-sum of $\Delta(\cG, \cS)$, i.e. half the sum of the characters of $\cS$ appearing in the adjoint action on $\cN$ weighted by their multiplicities; and $\tilde \rho$ the half-sum of $\Delta(\cG, \cT)$, so $\tilde\rho|_\cS = \rho$. 

Let $W$ be an irreducible algebraic representation of $\cG$ over $E$. Then $W^{\cN}$ is an
irreducible algebraic representation of $\cM$, and in particular $Z_{\cM}$ acts on $W^{\cN}$ via a character $\psi$. Let $\tilde \psi$ be the highest weight of $W^{\cN}$ with respect to $\cM \cap \cB$; then $\tilde\psi|_{Z_{\cM}} = \psi$. It is shown in \emph{op.cit.} that the element
\[ s_{\tilde\alpha}(\tilde\psi + \tilde \rho)|_{\cS} \in Y^\bullet\]
is independent of the choice of Borel subgroup $\cB$. Let $\chi$ be a character of $Z_{\tilde \scM}$ which is locally $\psi$-algebraic (that is, we may write $\chi = \theta\psi$, where $\theta$ is locally constant).

\begin{definition}[{\cite[Definition 4.4.3]{emerton-jacquet1}}]\label{non-crit-slope}
 We say $\chi = \theta \psi$ is of critical slope with respect to the representation $W^\cN$ if, for some simple positive root $\alpha \in \Delta(\cG, Z_\cM)$, the element
 \[ s_{\tilde\alpha}(\tilde\psi + \tilde \rho)|_{\cS} + \rho + \mathrm{slope}(\theta) \in (\Q \otimes_\Z Y^\bullet)\]
 lies in $(\Q  \otimes_\Z R)^{\ge 0}$. Otherwise, we say $\chi$ is of non-critical slope.
\end{definition}

Exactly as in \cite[\S 4.4]{emerton-jacquet1}, we deduce the following result:

\begin{theorem}\label{thm:smallslopes}
        Let $V$ be a locally analytic representation of $\tilde \scG$,
        and suppose that $V$ admits a $\tilde\scG$-invariant norm.
        Let $W$ be an irreducible algebraic representation of $\cG$, and let $\psi:\bZ_{\cM}\to \bG_{m}$ be
        the central character of $W^{\cN}$.
        Then for any character $\chi: Z_{\tilde \scM} \to E^\times$ which is locally $\psi$-algebraic
        and of non-critical slope with respect to $W^{\cN}$, the map
        \begin{equation}
                \label{isomorphism-eigenspace}
                J_{\cP}(V_{W-{\lalg}})[Z_{\tilde \scM} = \chi] \to J_{\cP}(V)_{W^{\cN}-\lalg}[Z_{\tilde \scM}=\chi]
        \end{equation}
        is an isomorphism.
\end{theorem}

This result is proved for algebraic groups in \cite[section 4.4]{emerton-jacquet1}, and the same proof
works for representations of metaplectic groups, using the key lemma \ref{lemma:+veslope2}. 

\begin{proof}
Injectivity of \eqref{isomorphism-eigenspace} follows from the fact that $J_{\cP}$ is left-exact.
We must prove surjectivity.
Note that by \cite[Prop 3.2.12]{emerton-jacquet1} we have isomorphisms
\[
        J_{\cP}(V)[Z_{\tilde \scM}^{+}=\chi]
        \cong
        (V^{\hat \scN_{0}})[Z_{\tilde \scM}^{+}=\chi],
        \quad
        J_{\cP}(V_{W-\lalg})[Z_{\tilde \scM}^{+}=\chi]
        \cong
        (V_{W-\lalg}^{\hat \scN_{0}})[Z_{\tilde \scM}^{+}=\chi].
\]
We must therefore show that every vector in $(V^{\hat \scN_{0}})_{W^{\cN}-\lalg}[Z_{\tilde \scM}^{+}=\chi]$ is
locally $W$-algebraic.

Suppose that $v$ is in $(V^{\hat \scN_{0}})_{W^{\cN}-\lalg}[Z_{\tilde \scM}^{+}=\chi]$ and is not locally $W$-algebraic. Since $v$ generates a direct sum of copies of $W^{\cN}$ under the action of $\Lie(\cP)$, we may assume that $v$ is annihilated by $\gn' = \Lie (\cN')$; that is, $v$ is a highest weight vector for $\gg$ with respect to $\gn'$, of weight $\tilde\psi$.
 
Hence the $U(\gg)$-submodule $(U\gg)\cdot v \subseteq V$ is a quotient of a Verma module with
highest weight $\psi$. Since by assumption $v$ is not locally algebraic, this quotient must be infinite-dimensional.
It follows from the Bernstein-Gelfand-Gelfand resolution of $W$ in terms of Verma modules that there must exist a positive simple root $\alpha$ such that $X_{-\alpha}^{m+1}v \ne 0$,
where $X_{-\alpha}$ is in the $-\alpha$ root space in $\gg$;
 $m=\langle\psi,\alpha^{\vee}\rangle$; and $\alpha^{\vee}$ is the corresponding coroot.
The calculation \cite[Prop. 4.4.4]{emerton-jacquet1} shows that $X_{-\alpha}^{m+1}v$ is in
$V^{\hat \scN_{0}}[Z_{\tilde \scM}^{+}=\gamma]$, where $\gamma = \alpha^{-m-1}\chi$.
Applying Lemma \ref{lemma:+veslope2} to $\gamma$ shows that if $V$ has a $\tilde\scP$-invariant norm, and $(V^{\hat \scN_0})[Z_{\tilde \scM}^+ = \gamma] \ne 0$, then $\rho+\slope(\gamma)$ must lie in $(\Q\otimes R^{\bullet})^{\ge 0}$. From this, we deduce that $\chi$ has critical slope, so we have a contradiction.
\end{proof}

\subsection{Emerton's eigenvariety machine}
\label{sect:eigenvar}

In this paragraph, our discussion becomes global again.
We therefore begin with a connected reductive group $\bG$ defined over an algebraic number field $k$
and a fixed prime $\gp$ of $k$ over $p$.
We define
\[
        \cG
        =
        \Rest^{k_{\gp}}_{\Qp} \left( \bG \times_{k} k_{\gp} \right).
\]
Thus $\cG$ is an algebraic group over $\Qp$ and we let $\scG$ be the group of $\Qp$-valued
points of $\cG$.
There is an isomorphism of groups $\scG=G_{\gp}$.
As before, we assume that we have a metaplectic extension $\tbG$ of $\bG$ by $\mu$,
and we define $\tilde\scG$ to be the central extension of $\scG$ obtained by identifying $\scG$
with $G_{\gp}$.

The local theory of the preceding paragraphs was, for a few technical reasons,
described only over $\Qp$. We shall apply this theory to the group $\tilde\scG$.

We make the assumption that $\cG$ is quasi-split over $\Qp$; note that this holds if and only if $\bG$ is quasi-split over $k_{\gp}$, which is well known to be true for all but finitely many primes $\gp$. Let $\cB$ be a Borel subgroup of $\cG$, and $\cT$ a Levi factor of $\cB$. Following our general notational conventions, we let $\scB$ and $\scT$ denote the $\Qp$-points of these, and $\tilde \scB$ and $\tilde \scT$ their preimages in $\tilde \scG = \tilde G_\gp$. 

Let $V$ be an essentially admissible locally analytic $\tilde \scG$-representation, equipped
with a commuting action of $\cH^{\gp}$.
Let $\cH^{\sph}$ be the spherical part of the Hecke algebra.
By functoriality, there is an action of $\cH^{\gp}$ on $J_{\cB}(V)$, which is an essentially
admissible representation of $\tilde \scT$.
Let $Z = Z_{\tilde \scT}$ be the centre of $\tilde \scT$.
Then $J_{\cB}(V)$ restricts to give an essentially admissible representation
of $Z$, and so there is a corresponding coherent rigid analytic sheaf $\scE$ on $\widehat Z$.
There is an action of $\cH^{\sph}$ on the sheaf $\scE$, and we let $\scA$ be the image of $\cH^{\sph}$
in the sheaf of endomorphisms of $\scE$.
We then define the Eigenvariety of $V$ to be the following a rigid analytic space
\[
        \Eig(V) = \Spec (\scA) \subset \widehat Z \times \Spec(\cH^{\sph}).
\]
A point $(\chi,\lambda)\in  \widehat Z \times \Spec(\cH^{\sph})$ is in $\Eig(V)$ if and only if
the $(Z=\chi,\cH^{\rm sph} = \lambda)$-eigenspace in $J_{\cB}(V)$ is non-zero. By construction, the sheaf $\scE$ is the push-forward to $\widehat Z$ of a sheaf on $\Eig(V)$, which we also denote by $\scE$. This sheaf $\scE$ is a sheaf of right $\cH^{\gp}$-modules, and the fibre of $\scE$ over a point $(\chi, \lambda)$ is isomorphic as a right $\cH^{\gp}$-module to the dual of the $(Z=\chi,\cH^{\rm sph} = \lambda)$-eigenspace in $V$.

\begin{theorem}{\mbox{~}}
 \begin{enumerate}
 \item[(i)] The map $\Eig(V) \to \widehat Z$ is finite. 
  \item[(ii)] If $V$ is admissible as a representation of $\tilde \scG$, then the map $\Eig(V)\to \check \gt$ has discrete fibres; in particular the dimension of $\Eig(V)$ is at most the dimension of $\cT$ over $\Qp$.
  \end{enumerate}
\end{theorem}

\begin{proof}
Part (i) is true by construction, since $\Eig(V)$ is defined as the relative spectrum of a coherent sheaf of algebras on $\widehat Z$. Part (ii) follows from the corresponding statement for the support of the coherent sheaf $\scE$ on $\widehat Z$, which is Theorem \ref{thm:discretefibres} above.
\end{proof}

Let us now take $V = \tH^n_{\varepsilon}(\hat K^{\gp}, E)_{\Qp-\la}$. This is admissible, so the preceding theorem applies to $\Eig(V)$. We suppose for the remainder of this section that the edge map criterion (Definition \ref{edgemapcriterion}) holds  for $(\tbG, \gp, \hat K^{\gp}, \varepsilon, n)$. Then for all algebraic representations $W$ of $\cG$ we have
an isomorphism of smooth $\tilde \scG\times \cH(K^{\gp})$-modules:
\[
        H^{n}_{\cl,\varepsilon}(\hat K_{\gp},W')
        =
        \Hom_{\gg} \left(W, V\right).
\]
We shall show that the eigenvariety $\Eig(V)$ interpolates the finite slope representations in
$H^{n}_{\cl,\varepsilon}(\hat K^{\gp},W')$.

Suppose $\pi$ is an absolutely irreducible representation of $\tilde \scG\times \cH^{\gp}$, which
appears as a subquotient of $H^{n}_{\cl,\varepsilon}(K^{\gp},W')$ for some irreducible algebraic representation $W$
of $\cG$, and suppose that $\pi_{\gp}$ embeds in $\ind_{Z \hat \scN}^{\tilde \scG}(\theta)$ for some smooth
character $\theta$ of the centre $Z$ of $\tilde \scT$.
The Hecke algebra $\cH^{\sph}$ acts on $\pi$ by a character $\lambda\in \Spec(\cH^{\sph})(\bar\Qp)$.
Let $\psi$ be the highest weight character of $W$ with respect to $\cB$, regarded as a character of $Z_{\gp}$.
Then the point $(\theta\psi,\lambda)\in \widehat Z \times \Spec(\cH^{\sph})$ is called a classical point.

\begin{theorem}
        Every classical point is in $\Eig(V)$.
\end{theorem}

\begin{proof}
We first note that $J_{\cB}$ is exact on the subcategory of admissible smooth representations of $\tilde \scG$.
This is because (i) it is constructed as a composition of two left-exact functors, and is therefore left-exact,
and (ii) by theorem \ref{smooth-jacquet}, the Jacquet functor coincides on smooth representations with
the coinvariants, which is right-exact.

By exactness, there is a sub-quotient of $J_{\cB}(H^{n}(\hat K^{\gp}, W'))$ on which
$Z\times \cH^{\sph}$ acts by $(\theta,\lambda)$.
The vector space $J_{\cB}( H^{n}(\hat K^{\gp}, W'))\otimes_{E} \bar \Qp$
 is an admissible smooth representation of $Z$, and is therefore a direct limit of finite dimensional
 $Z\times \cH^{\sph}$-modules.
Therefore the $(\theta,\lambda)$-eigenspace in $J_{\cB}( H^{n}(\hat K^{\gp}, W'))\otimes_{E} \bar \Qp$ is non-zero.
By Corollary \ref{loc-alg-jacquet} we deduce that the $(\theta\psi,\lambda)$-eigenspace
in $J_{\cB}(H^{n}(\hat K^{\gp}, W')\otimes W)$ is non-zero.
By the edge map criterion,
 the representation $H^{n}(\hat K^{\gp}, W')\otimes W$ is isomorphic to
the closed subspace of $W$-locally algebraic vectors in $V$.
By left-exactness, we deduce that the $(\theta\psi,\lambda)$-eigenspace in $J_{\cB}(V)$ is non-zero.
This implies that $(\theta\psi,\lambda)$ is in $\Eig(V)$.
\end{proof}

\begin{theorem}
 Let $(\theta\psi,\lambda)$ be a point of $\Eig(V)$, where $\theta$ is locally constant and $\psi$ is the highest weight of some algebraic representation $W$ of $\cG$. If $\chi = \theta\psi$ has non-critical slope, then $(\theta\psi, \lambda)$ is a classical point.
\end{theorem}

\begin{proof}
This is immediate from Theorem \ref{thm:smallslopes}, since $V = \tH^n_{\varepsilon}(\hat K^{\gp}, E)_{\Qp-\la}$ admits a $\tilde\scG$-invariant norm given by the gauge of the lattice $\tH^n_{\varepsilon}(\hat K^{\gp}, \cO_E) \subset \tH^n_{\varepsilon}(\hat K^{\gp}, E)$.
\end{proof}

\section{A \texorpdfstring{$p$}{p}-adic analytic Stone--von Neumann theorem}
\label{sect:stonevonneumann}

If the machinery of the previous section is applied in the case where the parabolic subgroup $\cP$ is a Borel subgroup, the Jacquet module is a representation of the preimage in $\tilde \scG$ of a maximal torus in $\scG$. This is a topological central extension of a commutative group, but need not itself be commutative. We therefore turn to the question of classifying locally analytic representations of metaplectic tori. 

In this section, we shall let $\scT$ be any topologically finitely generated abelian $\Qp$-analytic group. We define the \emph{rank} of $\scT$ to be the rank of the finitely-generated abelian group $\scT / \scT_0$, for any compact open subgroup $\scT_0 \subseteq \scT$. Let $\tilde\scT$ be a topological central extension of $\scT$ by $\mu$, so there is an exact sequence of topological groups
\[ 1 \to \mu \to \tilde\scT \to \scT \to 1.\]
Let us fix a character $\varepsilon: \mu \to E^\times$, where $E$ is (as above) a discretely valued closed subfield of $\mathbb{C}_p$. We shall restrict our attention to representations of $\tilde\scT$ on which $\mu$ acts via $\varepsilon$.
Without loss of generality, we suppose that $\varepsilon$ is injective, and in particular $\mu$ is cyclic.

Let ${Z}$ be the centre of $\tilde\scT$. This contains $\mu$, and therefore is the preimage of a subgroup $\scZ \subseteq \scT$. The group $\tilde\scT / {Z} \cong \scT / \scZ$ is a finite abelian group, of exponent dividing the order of $\mu$.
This quotient group is equipped with a non-degenerate, alternating bilinear form 
\[
        \Lambda^{2}(\tilde\scT / {Z})\to \mu
\]
given by $t\wedge u \mapsto [t,u]$; in particular, the index $[\scT : \scZ]$ is a square. If $\scA$ is a subgroup of $\scT$ containing $\scZ$, then the preimage $\tilde \scA$ is abelian if and only if $\scA / \scZ$ is an isotropic subspace of $\scT / \scZ$. We shall abuse notation and call such subgroups \emph{isotropic subgroups} of $\scT$.

For a locally $p$-adic analytic group $\scG$ whose centre is topologically finitely generated (such as all of the groups considered above), we let $\Repess(\scG)$ denote the category of essentially admissible locally analytic representations of $\scG$. Note that if $\scG$ is commutative, then $\Repess(\scG)$ is the opposite category of the category of coherent sheaves on the rigid-analytic space $\widehat \scG$.
For $\scG$ a subgroup of $\tilde\scT$ containing $\mu$, we let $\Repess(\scG)_\varepsilon$ denote the subcategory of representations on which $\mu$ acts via the character $\varepsilon$. If $\scG$ is commutative, then this category is anti-equivalent to the coherent sheaves on a subspace $\widehat{\scG}_\varepsilon \subseteq \widehat{\scG}$, which is a union of components of $\widehat{\scG}$.

\subsection{Irreducible representations}

We begin with a weak form of the Stone--von Neumann theorem, classifying the irreducible objects of $\Repess(\tilde\scT)_\varepsilon$. Let us choose a maximal isotropic subgroup $\scA \subseteq \scT$, as above.

\begin{lemma}\label{characters}
 Let $\chi$ be a continuous (hence locally $\Qp$-analytic) character of ${Z}$ restricting to $\varepsilon$, $I_\chi$ the set of $\overline{E}$-valued characters of $\tilde \scA$ extending $\chi$, and $\psi \in I_\chi$. Then the map $\tilde\scT / \tilde \scA \to I_\chi$ mapping $t$ to the character $a \mapsto \psi(t^{-1} a t)$ is a bijection. (In other words, $I_\chi$ is a torsor for $\tilde\scT / \tilde \scA$).
\end{lemma}

(Note that any $\psi \in I_{\chi}$ is defined over a finite, and hence complete and discretely valued, extension of $E$. Moreover, since ${Z}$ is open in $\tilde \scT$, any such $\psi$ is locally $\Qp$-analytic.)

\begin{proof}
Let $\psi_{t}$ be the character defined by $\psi_{t}(a)=\psi(t^{-1}a t)$.
We shall first show that the map $t \mapsto \psi_{t}$ is injective on $\tilde \scT/ \tilde \scA$.
Suppose $\psi_{t}=\psi_{u}$ for $t,u \in \tilde\scT$. We need to show that $t^{-1}u \in \tilde A$.
By definition we have for all $a\in \tilde A$:
\[
        \psi(t^{-1}at)=\psi(u^{-1}au).
\]
Hence the element $t^{-1}a t u^{-1}a^{-1}u$ is in the kernel of $\psi$.
Since $\scT$ is abelian, the image of this element in $\scT$ is the identity, so
$t^{-1}a t u^{-1}a^{-1}u$ is also in $\mu_{m}$.
Since the restriction of $\psi$ to $\mu_{m}$ is $\varepsilon$, which is assumed to be injective,
we deduce that $t^{-1}at=u^{-1}au$.
In other words $ut^{-1}$ commutes with every element $a\in \tilde \scA$. As $\tilde \scA$ is a maximal abelian subgroup, we must therefore have $ut^{-1} \in \tilde \scA$.

To prove surjectivity, we'll show that $\tilde \scT/ \tilde \scA$ and $I_{\chi}$ have the same number of elements.
The number of elements in $I_{\chi}$ is $|\tilde \scA / {Z}|$, which is
the same as $|\scA/\scZ|$. Since $\scA / \scZ$ is a maximal isotropic subspace of $\scT / \scZ$ with respect to the skew-symmetric form above, $\scT/\scA$ is identified with the Pontryagin dual of $\scA/\scZ$ and hence $|\scT / \scA| = |\scA / \scZ|$.
\end{proof}

\begin{proposition}\label{prop:vchi}
 Let $\tilde \scT$ be a topological central extension by $\mu$
 of a topologically finitely generated abelian locally $\Qp$-analytic group
 $\scT$, and let ${Z}$ be the centre of $\tilde \scT$. Fix an injective character
 $\varepsilon$ of $\mu$.

 Then for every locally $\Qp$-analytic $E$-valued character $\chi$ of ${Z}$ extending $\varepsilon$, there is a unique irreducible finite-dimensional locally $\Qp$-analytic representation $V_\chi$ of $\tilde\scT$ on an $E$-vector space having central character $\chi$. 
\end{proposition}

\begin{proof}
 Let $A_\chi$ be the algebra
 \[ E[\tilde \scT] / \langle z - \chi(z) : z \in {Z}\rangle.\]
 This is a finite-dimensional $E$-algebra of dimension $d^2$, where $d = |\scT / \scA| = |\scA / \scZ|$, and it is clear that any representation of $\tilde\scT$ on which ${Z}$ acts via $\chi$ is a module over $A_\chi$. Note that since ${Z}$ is open and has finite index in $\tilde\scT$, such a representation is essentially admissible if and only if it is finite-dimensional over $E$, or equivalently finitely-generated over $A_\chi$; thus the essentially admissible representations of $\tilde\scT$ with central character $\chi$ are precisely the finitely-generated $A_\chi$-modules.
 
 We claim that $A_\chi$ is a central simple $E$-algebra. It suffices to check this after any finite base extension, so let us choose a finite extension $E' / E$ sufficiently large that all characters $\psi \in I_\chi$ are defined over $E'$. 
 
 We consider $A_\chi$ as a representation of $\tilde \scA \times \tilde \scA$ via right and left translation. If $E' / E$ is a finite extension sufficiently large that all $\psi \in I_\chi$ have values in $E'$, we may decompose $A_\chi' = A_\chi \otimes_E E'$ as a direct sum of isotypical components $A^{(\psi_1, \psi_2)}_\chi$ for the action of $\tilde \scA \times \tilde \scA$, indexed by pairs $(\psi_1, \psi_2) \in I_\chi \times I_\chi$, and any $\tilde\scA \times \tilde\scA$-stable subspace of $A_\chi'$ is equal to the direct sum of its intersections with these isotypical subspaces.
 
So let $S$ be any two-sided ideal in $A_\chi'$.
It follows that $S$ is $\tilde\scA\times \tilde \scA$-invariant,
and the action of $\tilde\scA\times \tilde \scA$ on $S$ is obviously diagonalizable.
Hence if $S$ is non-zero, then it must have nontrivial intersection with $A^{(\psi_1, \psi_2)}_\chi$
for some pair $(\psi_1, \psi_2)$.
However, since $S$ is a two-sided ideal, and conjugation by $\tilde \scT$ permutes $I_\chi$ transitively, we deduce that it must have nontrivial intersection with \emph{all} of the subspaces $A^{(\psi_1, \psi_2)}_\chi$. Hence its dimension is at least $d^2$, which is the dimension of $A_\chi'$; thus $S = A_\chi'$. So we have shown that $A_\chi'$ is simple, from which it follows that $A_\chi$ is simple.
 
 Since $A_\chi$ is a central simple algebra, up to isomorphism there is a unique simple left $A_\chi$-module; and we define $V_\chi$ to be this representation.
\end{proof}

\begin{remark}
        Note that the dimension of $V_\chi$ is equal to $de$, where $e$ is the order of $A_\chi$ in the Brauer group of $E$.
        In particular, $V_{\chi}$ is absolutely irreducible if and only if $A_\chi$ is isomorphic to a matrix algebra,
        so $e = 1$.
\end{remark}

\begin{proposition}\label{prop:homnonzero}
 Let $W$ be any essentially admissible locally $\Qp$-analytic representation of $\tilde \scT$ on which $\mu$ acts via $\varepsilon$, and let $V_\chi$ be the representation constructed above. Then for any character $\chi: {Z} \to E^\times$, we have
 \[ \Hom_{{Z}}(\chi, W) \ne 0 \Leftrightarrow \Hom_{\tilde \scT}(V_\chi, W) \ne 0.\]
\end{proposition}

\begin{proof}
 By replacing $W$ with the closed $\tilde\scT$-stable subspace $W^{{Z} = \chi}$, it suffices to show that if ${Z}$ acts on $W$ via $\chi$, then there is a nonzero homomorphism of $\tilde\scT$-representations $V_\chi \to W$. But since $W$ is essentially admissible as a representation of ${Z}$, it must be finite-dimensional, and thus finitely-generated as an $A$-module where $A$ is the central simple algebra constructed above. Since every finitely-generated module over a central simple algebra is a direct sum of copies of the unique simple module, it follows that $\Hom_{A}(V_\chi, W) = \Hom_{\tilde\scT}(V_\chi, W)$ is nonzero.
\end{proof}

We now extend the definition of $V_\chi$ slightly. The continuous characters $\chi: {Z} \to E^\times$ extending $\varepsilon$ are precisely the absolutely irreducible objects of the category $\Repess({Z})_\varepsilon$. If $F / E$ is a finite extension and $\chi: {Z} \to F^\times$ is a continuous character, we regard $\chi$ as an object of $\Repess({Z})_{\varepsilon}$ via restriction of scalars. If the values of $\chi$ generate $F$ over $E$, this representation is irreducible (but not absolutely so, of course, unless $F = E$). This gives us a bijection between each of the following sets:
\begin{itemize}
 \item irreducible objects of $\Repess({Z})_\varepsilon$,
 \item points of the rigid space $\widehat{{Z}}_\varepsilon$ (in the sense of rigid geometry, i.e.~maximal ideals of its structure sheaf),
 \item $\Gal(\overline{E} / E)$-orbits of characters $\chi: {Z} \to \overline{E}^\times$ extending $\varepsilon$.
\end{itemize}

For each Galois orbit of characters $\chi$, applying Proposition \ref{prop:vchi} with the coefficient field $E$ replaced by the finite extension $F/E$ generated by the values of $\chi$, we see that there is a unique irreducible $F$-linear representation $V_\chi$ of $\tilde\scT$ with central character $\chi$; and since the values of $\chi$ generate $F$ over $E$, this is still irreducible when regarded as an $E$-linear representation via restriction of scalars. We define $V_\chi$ to be this representation, which clearly depends only on the $\Gal(\overline{E} / E)$-orbit of $\chi$. As a representation of ${Z}$, $V_\chi$ is isomorphic to a direct sum of copies of the object of $\Repess({Z})_{\varepsilon}$ constructed from $\chi$ above; we shall abuse notation slightly by writing ``$V_\chi$ has central character $\chi$'' even when $\chi$ is not defined over $E$.

\begin{corollary}\mbox{~}
        \begin{enumerate}
                \item[(a)]
                The map $\chi \mapsto V_\chi$ is a bijection between the set of irreducible objects of the categories $\Repess({Z})_\varepsilon$ and $\Repess(\tilde\scT)_\varepsilon$, which is uniquely characterised by the fact that $V_\chi$ has central character $\chi$.
                \item[(b)]
                For any $W \in \Repess(\tilde\scT)_\varepsilon$, there is a closed rigid-analytic subvariety $X$ of the character space
                $\widehat{{Z}}_\varepsilon$ such that the irreducible subrepresentations of $W$ are precisely the $V_\chi$ for $\chi \in X$.
                \item[(c)] The irreducible representation $V_\chi$ is absolutely irreducible if and only if $\chi$ is defined over $E$ and the central simple algebra $A_\chi$ of Proposition \ref{prop:vchi} is trivial in the Brauer group of $E$. 
        \end{enumerate}
\end{corollary}

\begin{proof}
 For part (a), the fact that the map exists and is injective is clear; so it suffices to check that it is surjective. Let $W$ be an irreducible object in $\Repess(\tilde\scT)_\varepsilon$. Applying Proposition \ref{prop:homnonzero} to $\Rest^{\tilde\scT}_{{Z}}(W)$, we see that there is some $\chi$ such that $\Hom_{\tilde\scT}(V_\chi, W) \ne 0$. Since $W$ is irreducible, we must have $W \cong V_\chi$.

For part (b), we simply take $X$ to be the support of the sheaf on $\widehat{{Z}}$ corresponding to the locally analytic representation $\Rest_{{Z}}^{\tilde \scT}(W)$ of ${Z}$. By construction, the points of $X$ are precisely the $\chi$ such that $\Hom_{{Z}}(\chi, W) \ne 0$, and Proposition~\ref{prop:homnonzero} shows that these are precisely the irreducible $\tilde\scT$-subrepresentations of $W$.

Part (c) is immediate from the construction of $V_\chi$ and the remark on its dimension above.
\end{proof}

\subsection{Tame isotropic subgroups}

 We now show that the results of the previous subsection can be strengthened, under a mild additional hypothesis on $\tilde \scT$.

\begin{definition}
        Let $\scA \supset \scZ$ be a maximal isotropic subgroup.
        We say $\scA$ is \emph{tame} if $\scZ \scA_{\tors} = \scA$,
        where $\scA_{\tors}$ is the torsion subgroup of $\scA$.
\end{definition}

We have used the word ``tame'' to indicate the analogy to the tame symbol, as the
following example illustrates.

\begin{example}
        Let $k_{\gp}$ be a finite extension of $\Qp$ containing a primitive $m$-th root of unity.
        Let $\scT = k_\gp^\times$ and let $\mu_m$ be the group of all $m$-th roots of unity in $k_{\gp}$.
        Then we may define an extension $\tilde\scT$ of $\scT$ by $\mu_m$ by setting
        \[
                \tilde\scT
                =
                \{ (x, \zeta) : x \in k_\gp^{\times}, y \in \mu_m\}
        \]
        with the group law
        \[
                (x, \zeta) (y, \xi) = (xy, \zeta \xi (x, y)_{m}),
        \]
        where $(x, y)_m$ is the $m$-th power Hilbert symbol in $k_{\gp}$.
        
        One finds that
        \[
                \scZ
                =
                \{ x \in k_{\gp}^{\times} : x^2 \in k_{\gp}^{\times m}\}.
        \]
        Let us define $n = m$ if $m$ is odd, and $n = m/2$ if $m$ is even.
        Then, since we clearly have $\mu_{2n} \subseteq k_{\gp}^{\times}$
        and hence $-1 \in k_{\gp}^{\times n}$, we see that $\scZ = k_{\gp}^{\times n}$.
        
        If $n$ is not a multiple of $p$, then we have $[\scT : \scZ] = [k_{\gp}^\times : k_{\gp}^{\times n}] = n^2$. 
        If we take $\scA = \pi^{\Z } \cO^{\times n}$, where $\cO$ is the ring of integers of $k_\gp$ and $\pi$
        is a uniformizer, then $\scA/\scZ$ is isotropic (as any lift of $\pi$ to $\tilde\scT$ clearly commutes
        with itself). As $[\scA : \scZ] = n$, $\scA$ is maximal.

        Since we have assumed $p$ does not divide $n$, we have $1 + \pi \cO \subseteq \cO^{\times n}$;
        hence every element of $\scA / \scZ \cong \cO^{\times} / \cO^{\times n}$ has a representative that is a
        Teichmuller lift of an element of the residue field, and thus in particular lies in $\scA_{\tors}$.
        Hence $\scA$ is tame.
        On the other hand, the (possibly more natural) choice $\scA' = \pi^{n\Z } \cO^\times$
        is a maximal isotropic subgroup which is not tame.
        
        If on the other hand $p$ is a factor of $n$, then there is no tame maximal isotropic subgroup.
\end{example}

It is clear that if $p \nmid m$ and the rank of $\scT$ is 0 or 1, there must exist a tame maximal isotropic subgroup. On the other hand, if $\scT$ is the discrete group $\Z^2$ and $\tilde \scT$ is the unique non-split extension of $\scT$ by $\pm 1$, then there are precisely three maximal isotropic subgroups of $\tilde \scT$ and none of these are tame.

Let $R$ denote the finite \'etale cover $\widehat{\tilde \scA} \to \widehat{{Z}}$ given by restriction of characters.

\begin{proposition}
 If $\scA$ is tame, then $R$ maps every component of $\widehat{\tilde \scA}$ isomorphically to a component of $\widehat{{Z}}$.
\end{proposition}

\begin{proof}
 Recall that if $\scH$ is an abelian locally analytic group, the geometrically connected components of $\widehat \scH$ correspond bijectively with the characters of $\scH_{\mathrm{tors}}$. 

By \cite[Proposition 6.4.1]{emerton-memoir}, we can (non-uniquely) decompose $\tilde \scA$ and ${Z}$ as products $\tilde\scA = \tilde \scA_{\mathrm{tors}} \times \tilde \scA_\infty$ and ${Z} = {Z}_{\mathrm{tors}} \times {Z}_\infty$.
Moreover, it is clear that we may do this in such a way that ${Z}_{\infty} \subseteq \tilde \scA_{\infty}$.
The assumption that $\scA$ is tame, so that $\tilde \scA = {Z} \tilde \scA_{\mathrm{tors}}$, implies that in fact ${Z}_\infty = \tilde\scA_\infty$.

Since the contravariant functor $\widehat{(-)}$ takes direct products of groups to fibre products of rigid spaces, the map $\widehat{\tilde \scA} \to \widehat{{Z}}$ is obtained by taking the fibre product of the map of finite rigid spaces $\widehat{\tilde \scA_\mathrm{tors}} \to \widehat{{Z}_\mathrm{tors}}$ with the connected rigid space $\widehat{{Z}_\infty}$, from which the result is clear.
\end{proof}

\begin{corollary}
 For all sufficiently large $E$, there exists a map $S : \widehat{{Z}} \to \widehat{\tilde \scA}$ of rigid spaces over $E$ which is a section of $R$.
\end{corollary}

\begin{proof}
 By the same argument as in the previous proposition, it suffices to show that the map $\widehat{\tilde \scA_\mathrm{tors}} \to \widehat{{Z}_\mathrm{tors}}$ admits a section; this is clear, since both spaces are finite unions of points.
\end{proof}

\begin{remark}
 It suffices to assume that $\zeta_r \in E$, where $r$ is the exponent of $\tilde \scA_{\mathrm{tors}}$. If $E$ does not contain enough roots of unity then such a section may well not exist. Note also that the choice of the section $S$ is highly non-canonical.
\end{remark}

\begin{theorem}
 If $\tilde\scT$ contains a tame maximal commutative subgroup, and $E$ is sufficiently large, then there is an equivalence of categories 
 \[ \Repess({Z})_\varepsilon \stackrel\sim\to \Repess(\tilde\scT)_\varepsilon.\]
\end{theorem}

\begin{proof}
 Let $V$ be an object of $\Repess({Z})$. We decompose $V$ in the form
 \[ V = \bigoplus_{\chi} V^\chi\]
 where the sum is over the characters of ${Z}_{\mathrm{tors}}$. We let $F_\scA(V)$ be the representation of $\scA$ which is isomorphic to $V$ as a representation of ${Z}_\infty$, but with $\tilde \scA_{\mathrm{tors}}$ acting on the summand $V_\chi$ by the extension of $\chi$ to a character of $\tilde \scA_{\mathrm{tors}}$ determined by the section $S$ above.

It is clear that $F_\scA$ is a functor $\Repess({Z})_\varepsilon \to \Repess(\tilde \scA)_\varepsilon$. Composing this with the functor $\Ind_{\tilde \scA}^{\tilde\scT} : \Repess(\tilde \scA) \to \Repess(\tilde\scT)$ gives a functor $F : \Repess({Z})_\varepsilon \to \Repess(\tilde\scT)_\varepsilon$.

We construct an inverse functor $G : \Repess(\tilde\scT) \to \Repess({Z})$ as follows. Restriction to $\tilde \scA$ gives a functor $\Repess(\tilde\scT) \to \Repess(\tilde \scA)$. We may decompose $\Repess(\tilde \scA)$ as a direct sum of subcategories corresponding to the characters of $\tilde \scA_\mathrm{tors}$. 
 
The choice of section $S$ above determines a subset $I \subset \widehat{\tilde \scA_\mathrm{tors}}$. The functor $\pr_{I} : \Repess(\tilde \scA) \to \Repess(\tilde \scA)$ given by $V \mapsto \bigoplus_{\psi \in I} V^\psi$ is well-defined.
We define $G = \operatorname{Res}_{\tilde \scA}^{{Z}} \circ \pr_I \circ \operatorname{Res}_{\tilde\scT}^{\tilde\scA}$, which clearly defines a functor $\Repess(\tilde\scT)_\varepsilon \to \Repess({Z})_\varepsilon$.

 We claim that these functors are inverse to each other. Let us first consider the composition $G \circ F : \Repess({Z})_\varepsilon \to \Repess({Z})_\varepsilon$. We note that there is a natural isomorphism of functors $\Repess(\tilde \scA) \to \Repess(\tilde \scA)$ between the functors $\operatorname{Res}_{\tilde\scT}^{\tilde \scA} \circ \Ind_{\tilde \scA}^{\tilde\scT}$ and the functor mapping $V$ to $\bigoplus_{t \in \tilde\scT / \tilde \scA} V^t$. Using lemma \ref{characters}, we deduce that $G \circ F$ is naturally isomorphic to the identity functor.

 On the other hand, let $W$ be an object of $\Repess(\tilde\scT)_\varepsilon$. Since all our functors commute with direct sums, we may as well assume that ${Z}_{\mathrm{tors}}$ acts on $W$ via a character $\chi$ (extending $\varepsilon$). Then $\psi = S(\chi)$ is a choice of extension of $\chi$ to $\tilde \scA_\mathrm{tors}$. We find that
 \[
        F(G(W))
        =
        \Ind_{\tilde A}^{\tilde\scT} \left( W^{[\tilde \scA_\mathrm{tors} = \psi]} \right).
\]
 We must construct a natural transformation between this and the identity functor. Let $\psi_1, \dots, \psi_s$ be the conjugates of $\psi$, and fix $t_1, \dots, t_s \in \tilde\scT$ such that $\psi(t_i^{-1} a t_i) = \psi_i(a)$. Then the map $f \mapsto \sum t_i f(t_i^{-1}) : F(G(W)) \to W$ gives a natural transformation between $F \circ G$ and the identity functor.
\end{proof}

\begin{remark}
If $\scT$ has no tame maximal isotropic subgroup, then the category
 $\Repess(\tilde\scT)_{\varepsilon}$ may genuinely fail to be isomorphic to $\Repess({Z})_{\varepsilon}$, even if the scalar field $E$ is large.
 
For instance, let $\tilde \scT$ be the extension of $\Z^2$ by $\pm 1$ mentioned above, and $\varepsilon$ the nontrivial character. If we set $\mathcal{D} = E[\tilde \scT] \otimes_{E[{Z}]} \cC^{\an}(\widehat{{Z}}_{\varepsilon}, E)$, then $\mathcal{D}$ is a Fr\'echet--Stein algebra, and $\Repess(\tilde \scT)$ is precisely the opposite category of coadmissible $\mathcal{D}$-modules; but $\mathcal{D}$ is a non-trivial central simple algebra of rank 4 over $C^{\an}(\widehat{{Z}}_{\varepsilon}, E)$, and hence the categories of coadmissible modules over these rings are not equivalent.

 Nonetheless, restriction of representations certainly gives a functor $\Repess(\tilde \scT)_{\varepsilon} \to \Repess({Z})_{\varepsilon}$, and one deduces that for any $V \in \Repess(\tilde \scT)_\varepsilon$ there is a coherent sheaf on $\widehat{{Z}}_{\varepsilon}$ whose support consists precisely of the characters of ${Z}$ appearing in $V$; the disadvantage is that since the restriction functor is not full, morphisms between such sheaves do not necessarily correspond to morphisms of $\tilde \scT$-representations.
\end{remark}

We may summarise the discussion as follows:

\begin{theorem}\label{thm:stonevonneumann}
        Let $\tilde \scT$ be a topological central extension by $\mu$
        of a topologically finitely generated abelian locally $\Qp$-analytic group
        $\scT$, and let ${Z}$ be the centre of $\tilde \scT$. Fix an injective character
        $\varepsilon$ of $\mu$.
        \begin{enumerate}
                \item
                For every $\Gal(\overline{E} / E)$-orbit of continuous characters $\chi: {Z} \to \overline{E}^\times$ extending $\varepsilon$, there is a unique irreducible locally analytic representation $V_\chi$
                of $\tilde \scT$ having central character $\chi$.
                \item
                If $W$ is an essentially admissible locally analytic representation of
                $\tilde \scT$ on which $\mu$ acts via $\varepsilon$, then there is a closed
                rigid-analytic subvariety $\operatorname{Supp}(W)$ of $\widehat{{Z}}$
                having the property that $\Hom_{\tilde\scT}(V_\chi, W) \ne 0$ if and  only if $\chi \in
                \operatorname{Supp}(W)$.
                \item
                If, moreover, $\scT$ has a tame maximal isotropic subgroup, then
                the category of locally analytic representations of $\tilde \scT$ is
                anti-equivalent to the category of rigid-analytic sheaves on
                $\widehat{{Z}}_\varepsilon$, and $\operatorname{Supp}(W)$
                is the support of the sheaf corresponding to $W$.
         \end{enumerate}
\end{theorem}

\begin{theorem}
        Let $\bG$ be semi-simple, simply connected and split over $k_{\gp}$
        and let $\tbG$ be the canonical metaplectic extension of $\bG$ by $\mu_{m}$.
        Let $\bT$ be a maximal split torus in $\bG$, and let $\scT$ and $\tilde\scT$ be
        as before.
        If $\gp$ does not divide $m$, then there is a tame maximal isotropic
        subgroup of $\scT$.
\end{theorem}

\begin{proof}
        The torus $\bT$ is split, so we have an isomorphism
        $\bT \cong \G_{m}^{n}$ for some $n$.
        We shall regard elements of $\scT$ as vectors $\underline{t}=(t_{i})_{i=1}^{n}$ with $t_{i}\in k_{\gp}^{\times}$.
        By \cite[Lemme II.5.4 and Lemme II.5.8]{matsumoto} we know that the commutator has
        the form
        \[
                [\underline{t},\underline{u}]
                =
                \prod_{i,j} (t_{i},u_{j})_{\gp,m}^{m(i,j)},
        \]
        where $m(i,j)$ are certain integers.
        One easily shows that $\scZ$ is the set of elements $\underline{t}$ satisfying for
        $j=1,\ldots,n$ the following relation:
        \[
                \prod_{i=1}^{n} t_{i}^{m(i,j)}
                \in
                k_{\gp}^{\times m}.
        \]
        Clearly, $\scZ$ contains $\scT^{m}$.
        Define a subgroup $\scA$ of $\scT$ by the relations
        \[
                \prod_{i=1}^{n} t_{i}^{m(i,j)}
                \in
                \cO^{\times} \cdot k_{\gp}^{\times m},
                \qquad
                j=1,\ldots, n.
        \]
        It follows that $\scA= \scZ \cdot (\cO^{\times})^{n}$.
        We claim that $\scA$ is a tame maximal isotropic subgroup of $\scT$.
        
        We first show that $\scA$ is isotropic.
        The elements of $\scA$ are, up to an element of $\scZ$, in the subgroup $(\cO^{\times})^{n}$,
        and so it suffices to show that $(\cO^{\times})^{n}$ is isotropic.
        This in turn follows from the fact that $(-,-)_{\gp,m}$ is the tame symbol (as $\gp$ does not divide $m$)
        and so is trivial on $\cO^{\times} \times \cO^{\times}$.
        
        We next show that $\scA$ is maximal.
        Assume that $\underline{u}$ satisfies $[\underline{u},\underline{t}]=1$ for all $\underline{t}\in \scA$.
        This implies for each $j$ and every element $t\in \cO^{\times}$ the relation
        \[
                \prod_{i=1}^{n} (t,u_{j})_{\gp,m}^{m(i,j)}
                =
                1.
        \]
        Hence $\prod_{i=1}^{n} u_{j}^{m(i,j)}$ is an element of $\cO^{\times}\cdot k_{\gp}^{\times m}$, and
        so $\underline{u}$ is in $\scA$.
        
        Finally, we show that $\scA$ is tame.
        To see this, it is sufficient to note that the cosets of $\cO^{\times}/\cO^{\times m}$ have representatives which are roots of unity (Teichm\"uller representatives).
        Again, we have used the fact that $\gp$ does not divide $m$, and so every element in $\cO^{\times}$
        which is congruent to $1$ modulo $\gp$ is an $m$-th power.
\end{proof}

\subsection{Implications for eigenvarieties}

We now return to the situation of section \ref{sect:eigenvar}. The group $\tilde \scT$ is an extension of the abelian, topologically finitely generated $p$-adic analytic group $\scT$ by $\mu$, so we may describe its essentially admissible representations by means of Theorem \ref{thm:stonevonneumann}. As before, we write ${Z}$ for the centre of $\tilde \scT$.

Recall that we have shown that, for any essentially admissible $\tilde \scG$-representation $V$ equipped with an action of $\cH^\gp$, there exists a rigid-analytic subspace $\Eig(V) \subseteq \widehat{{Z}} \times \Spec(\cH^\sph)$, and a coherent sheaf of right $\mathcal{H}^\mathfrak{p}$-modules $\scE$ on $\Eig(V)$, such that the fibre of $\scE$ at a point $(\chi, \lambda) \in \widehat{Z}$ is canonically isomorphic as an $\mathcal{H}^\mathfrak{p}$-module to the dual of the $({Z} = \chi, \cH^{\sph} = \lambda)$ eigenspace of $V$.

As before, let $\varepsilon$ be an injective character of $\mu$. We write $\Eig(V)_\varepsilon$ for the intersection of $\Eig(V)$ with $\widehat{{Z}}_\varepsilon \times \Spec(\cH^\sph) \subset \widehat{{Z}} \times \Spec(\cH^\sph)$. From Theorem \ref{thm:stonevonneumann}, we deduce:

\begin{corollary}\mbox{~}
 \begin{enumerate}
  \item A point $(\chi, \lambda) \in \widehat{{Z}}_\varepsilon \times \Spec(\cH^\sph)$ lies in $\Eig(V)_\varepsilon$ if and only if 
  \[ \Hom_{\tilde \scT}(V_{\chi}, V)[\cH^{\sph} = \lambda] \ne 0.\]
  \item If $\tilde \scT$ has a tame maximal isotropic subgroup, we may construct a coherent sheaf $\scE' \subseteq \scE$ on $\widehat{{Z}}_\varepsilon $ with an action of $\cH^\gp$, such that the fibre of $\scE'$ at $\chi$ is canonically isomorphic to the dual of $\Hom_{\tilde \scT}(V_{\chi}, V)[\cH^{\sph} = \lambda]$.
 \end{enumerate}
\end{corollary}

\providecommand{\bysame}{\leavevmode\hbox to3em{\hrulefill}\thinspace}
\providecommand{\MR}[1]{}
\renewcommand{\MR}[1]{%
 MR \href{http://www.ams.org/mathscinet-getitem?mr=#1}{#1}.
}
\newcommand{\articlehref}[2]{\href{#1}{#2}}

\end{document}